\documentclass{article}
\usepackage[a4paper]{geometry}
\usepackage{amsfonts,bm,doi,enumitem,fancyhdr,graphicx,amssymb,amsthm,amsmath,mathrsfs,mleftright,sectsty,subcaption,tikz,tikz-3dplot,titlesec,wrapfig,xcolor}
\usepackage{makecell}
\usepackage{longtable}

\usepackage[most]{tcolorbox}

\usepackage[markup=underlined,todonotes={textsize=footnotesize}]{changes}
\usepackage{color,soul}
\definechangesauthor[color=green]{M}
\definechangesauthor[color=red]{O}
\definechangesauthor[color=red]{A}

\newtheorem{theorem}{Theorem}[section]

\newtheorem{proposition}[theorem]{Proposition}
\newtheorem{corollary}[theorem]{Corollary}

\theoremstyle{remark}
\newtheorem{remark}[theorem]{Remark}
\newtheorem*{remark*}{Remark}
\theoremstyle{definition}
\newtheorem{definition}[theorem]{Definition}
\newtheorem{example}[theorem]{Example}

\newtheorem{problem}{Problem}

\newtheorem*{excont}{Example \continuation}
\newcommand{\continuation}{??}
\newenvironment{continueexample}[1]
{\renewcommand{\continuation}{\ref{#1}}\excont[continued]}
{\endexcont}

\newtcolorbox[number within=section, use counter=theorem]{mybox}[2][]{%
	title=Algorithm~\thetcbcounter: #2, 
	#1,
	breakable,
	colback=black!5!white,
	colframe=black
}

\newtcolorbox{mybox2}[2][]{%
	title=#2,
	breakable,
	colback=black!10!white,
	colframe=black
}

\def \r{\mathbb R}

\def \z{\mathbb Z}
\def \Z{\mathbb Z}

\DeclareMathOperator{\conv}{conv} \DeclareMathOperator{\iv}{lV}
\DeclareMathOperator{\isin}{lsin}
\DeclareMathOperator{\lsin}{lsin}
\DeclareMathOperator{\icos}{lcos}

 \DeclareMathOperator{\id}{ld}

\DeclareMathOperator{\il}{l\ell} 
 
\DeclareMathOperator{\aff}{Aff}

\DeclareMathOperator{\g}{GL} 
\DeclareMathOperator{\GL}{GL} 
\DeclareMathOperator{\sail}{sail} 

\def \({\langle}
\def \){\rangle}

\title{Geometry of multidimensional Farey summation algorithm and frieze patterns}
\author{Oleg Karpenkov and Matty van Son}
\usepackage[style=numeric-comp,natbib=true,backend=biber,maxnames=4,minnames=1]{biblatex}
\AtEveryBibitem{
	\clearfield{issn} 
	\clearfield{doi} 
	\clearfield{isbn} 
	
	\ifentrytype{online}{}{
		\clearfield{url}
	}
}

\renewbibmacro{in:}{}
\makeatletter
\def\keywords{\xdef\@thefnmark{}\@footnotetext}
\makeatother

\begin{document}

	\maketitle
	
	\begin{abstract}
		In this paper we develop a new geometric approach to subtractive continued fraction algorithms in high-dimensions. 
		We adapt a version of Farey summation to the geometric techniques proposed by F.~Klein in 1895.
		More specifically we introduce Farey polyhedra
		and their sails that generalise respectively Klein polyhedra and their sails, and show similar duality properties of the Farey sail integer invariants.
		The construction of Farey sails is based on the multidimensional generalisation of the Farey tessellation provided by a modification of the continued fraction algorithm introduced by R.~W.~J.~Meester. We classify Farey polyhedra in the combinatorial terms of prismatic diagrams.
		Prismatic diagrams extend boat polygons introduced by S.~Morier-Genoud and V.~Ovsienko in the two-dimensional case.
		As one of the applications of the new theory we get a multidimensional version of  Conway-Coxeter frieze patterns. We show that multidimensional frieze patterns satisfy generalised Ptolemy relations.  
	\end{abstract}
	
	\tableofcontents
	
	\vspace{5mm}

	\section{Introduction}

	This paper is dedicated to the study of multidimensional continued fractions for subtractive algorithms based on integer invariants of their {\it sails}. (Sails are polyhedral surfaces associated to continued fractions.)
	We introduce \textit{Farey polyhedra}, generalisations of Klein polyhedra in two-dimensions, and study their properties and combinatorics.
	Finally we define a notion of higher-dimensional frieze patterns, analogous to Conway-Coxeter frieze patterns.
	
	\vspace{2mm}
	
	This section is organised as follows. We start with a brief introduction to the subject areas in Subsection~\ref{subsection: history}.
	We discuss our main contributions and provide an overview of the organisation of the paper in Subsection~\ref{subsection: main contributions}.

	\subsection{History and background}
	\label{subsection: history}
	
	\noindent
	{\bf Multidimensional continued fractions.}
	The question on generalisation of continued fractions to the multidimensional
	case was raised for the first time in 1868 by C.~G.~J.~Jacobi
	(see~\cite{Jacobi1868}). One of the first generalisations of continued fractions
	was proposed by F.~Klein~\cite{Klein1896,Klein1895} in 1895. F.~Klein
	considered the cones in three-dimensional space; he introduced polyhedral
	surfaces that are now called {\it sails} (the boundaries of the convex hulls of all
	integer points inside the cone). Nearly 30 years ago V.~I.~Arnold initiated a
	detailed study of the geometric and combinatoric structures of Klein's polyhedra
	(see e.g. in~\cite{Arnold1998,Arnold2002}]). His suggestion was to examine
	the geometry and combinatorics of continued fractions via integer lattice
	invariants such as integer congruence classes
	of the faces of the sails, their quantities and
	frequencies, integer angles between the faces, integer distances, volumes, and
	so on.
	\vspace{1mm}
	
	H. Tsuchihashi~\cite{Tsuchihashi1983} found the connection between periodic
	multidimensional continued fractions and multidimensional cusp singularities. 
	J.-O. Moussafir
	in~\cite{Moussafir2000} and O. German in~\cite{German2002} studied the
	relationship between the sails of multidimensional continued fractions and Hilbert
	bases. Statistical properties of sails were studied by M.~L.~Kontsevich and Yu.~M.~Suhov in~\cite{Kontsevich1999}. 
	Multidimensional continued fractions appear in rigidity theory~\cite{KMMS2024}
	and toric geometry~\cite{MW2024}. 
	Some examples of the periodic
	multidimensional sails were calculated in the
	papers~\cite{Korkina1996,Korkina1995} by E. Korkina, \cite{Lachaud1993} by G. Lachaud,
	\cite{Bryuno1994, Parusnikov2000} by A.~D.~Bruno and V.~I.~Parusnikov, and the first author~\cite{Karpenkov2004,Karpenkov2006} (see also~\cite{Briggs,karpenkov-book-v2}). For the classical theory of regular continued
	fractions we refer to~\cite{Khinchin1961}.
	\vspace{1mm}
	
	Another famous generalisation of continued fractions was proposed in~\cite{Perron1907,Perron1913} by O.~Perron in 1907.
	O.~Perron invented an algorithm that produces approximations of a fixed direction (similarly to the continued fraction algorithm).
	Later his algorithm was referred to as the {\it Jacobi-Perron algorithm}. 
	The Jacobi-Perron algorithm was further modified to numerous {\it subtractive algorithms}. 
	Here we would like to mention the ordered Jacobi–Perron algorithm (OJPA)~\cite{Podsypanin1977,Hardcastle2000},
	Brun's algorithm~\cite{Brun1958,Berthe2018},
	and the fully subtractive algorithm by F.~Schweiger~\cite{Schw-2000}.
	\vspace{1mm}
	
	Technically speaking, Klein's approach to Jacobi's question works with cones rather then with a single direction, distinct from all subtractive algorithms. 
	For this reason there is no straightforward matching between their properties. 
	It is worth mentioning that in the classical two-dimensional case both Klein's and Perron's theories provide the same continued fractions.

	\vspace{2mm}
	
	\noindent
	{\bf Hermite's problem.}
	For the last 100 years subtractive algorithms and Klein's polyhedra were competing techniques to approach various mathematical problems. Let us consider one of such problems.
	{\it Hermite's problem} of 1848 asks 
	to find a comprehensive description of cubic irrationalities (i.e. the roots of cubic irreducible polynomials over $\mathbb Q$ with integer coefficients) 
	in terms of some eventually periodic sequences. The answer to this question would generalise 
	Lagrange's theorem on the periodicity of regular continued fractions for quadratic irrationalities.
	\vspace{1mm}
	
	Klein polyhedra are doubly-periodic for triples of conjugate vectors, however the constructive combinatorial description of such periods is not yet known. V.~I.~Arnold formulated several famous open problems on the quotient tori decompositions of fundamental domains that lead to the generalisation of the Lagrange's theorem on classical periodicity to the multidimensional case (see, e.g., in~\cite{karpenkov-book-v2}). 
	Partial results in this directions where obtained in papers~\cite{Lachaud1993,German2008}.
	\vspace{1mm}
	
	Contrary to the situation for Klein polyhedra, the description of periods 
	for subtractive algorithms is simple, as it is conjectured that most of the subtractive algorithms  are not necessarily periodic for cubic vectors.
	(For further discussions see, e.g.
	Jacobi's last theorem for the Jacobi-Perron algorithm 
	in Section~27.4 of~\cite{karpenkov-book-v2}.)
	It is only recently that the first periodic subtractive algorithms were constructed~\cite{karpenkov2021,karpenkov2022}.
	It is remarkable that the periodicity proof of the $\sin^2$-algorithm of~\cite{karpenkov2021} (the only subtractive algorithm which is proven to be periodic for cubic irrationalities) substantially involves the periodicity of Klein polyhedra.

	\vspace{2mm}
	
	\noindent
	{\bf Geometry of subtractive algorithms.}
	The first steps in studying the geometry of subtractive algorithms were done 
	in 2001 when T.~Garrity 
	introduced triangle sequences~\cite{Garrity2001}.
	Together with his coauthors he proved several dynamical results and a criterion on convergence in~\cite{Garrity2005}.
	In another paper~\cite{Garrity2004} he studied the Farey partition
	and generalises the Minkovskii ?-function (Minkowski question mark function) for it. 
	Later in 2008 G.~Panti considered another triangle map in~\cite{Panti2008}. 
	\vspace{1mm}
	
	We aim to develop new techniques of Klein polyhedra adjusted to subtractive algorithms. 
	We pick one particular subtractive algorithm that suits our purposes the best: {\it the Farey summation algorithm}, corresponding to the Farey partition introduced~\cite{Garrity2004}.
	This algorithm is based on the most straightforward generalisation of Farey summation to the higher-dimensional cases. 
	
	\vspace{1mm}
	
	This algorithm has remarkable properties.
	It admits a natural generalisation of the nose stretching algorithm
	(see~\cite{Arnold2002} for the classical nose stretching algorithm) and a remarkable convergence set (studied in detail in~\cite{KM1995}). 
	Note that the dual to the Farey summation algorithm is the Meester algorithm introduced  in 1989 by R.~W.~J~Meester in~\cite{Meester1989}. Namely the Meester algorithm acts as the Gauss map for the Farey summation algorithm. 
	As follows from the main theorem of~\cite{KM1995} (on Meester algorithm), the infinite sequences of cones produced by the Farey summation algorithm  do not converge to a single ray almost everywhere.
	As one of the consequences we get that Jacobi's last theorem does not hold for the Farey summation algorithm (there are cubic irrationalities with non-periodic continued fractions).
	\vspace{1mm}
	
	Informally speaking the Farey summation algorithm is the simplest algorithm from the geometric perspective. This is partially due to its rigid version of nose stretching algorithm, see Subsection~\ref{subsubsection: nose stretching} below.
	Note that the proposed geometric techniques can be generalised for other subtractive algorithms. All but a few of the properties will have straightforward generalisations. 
	
	\vspace{2mm}
	
	\noindent
	\textbf{Frieze patterns and triangulated polygons.}
	In the classical two-dimensional case {\it frieze patterns} are tables of numbers introduced by H.~Coxeter~\cite{Coxeter1970} and studied together with J.~Conway in~\cite{ConwayCoxeter1973}.
	In the latter paper a correspondence between frieze patterns and triangulations of convex polygons was found, which in turn has provided a connection with the more recent study of cluster algebras (see, e.g. in~\cite{moriergenoud2015,Schiffler2012}).
	
	\vspace{1mm}

	The correspondence may be understood in the following two ways. First, counting the number of triangles incident to each vertex in a triangulation provides a sequence that uniquely defines the frieze pattern up to cyclic permutations of the elements in the sequence.
	Second, by embedding the polygon triangulation to the Farey complex in the hyperbolic plane, one associates rational numbers to the vertices of the original polygon.
	There exists an elegant formula connecting rational numbers at pairs of vertices to each element of the frieze pattern corresponding to the triangulation, see the paper~\cite{short2022} by I.~Short for more details.
	Frieze patterns can be thought of as encoding the combinatorics of triangulated polygons.
	
	\vspace{1mm}
	
	In this paper we introduce prismatic diagrams that serve as an important combinatorial invariant of Farey polyhedra. 
	We use the combinatorics of prismatic diagrams to define three-dimensional frieze patterns where each element is associated to a pair of vertices on the prismatic diagram. 
	Similar to the two-dimensional case, three-dimensional frieze patterns admit a version of the Ptolemy relation for pairs of faces in a prismatic diagram.

	\subsection{Main results}
	\label{subsection: main contributions}
	
	In this paper we generalise Klein polyhedra and their sails to higher-dimensional cases and study their properties.
	In particular we introduce frieze patterns in higher-dimensions.
	
	\vspace{2mm}
	
	\noindent
	\textbf{Definition of Farey polyhedra.}
	Our generalisation of Klein polyhedra is provided by the following four basic algorithmic definitions: 
	\begin{itemize}
		\item \textbf{Algorithm~\ref{FareyTess-def}:} first we generate a tessellation of $\z^n$ that is equivalent to the tessellation of the hyperbolic plane by the Farey graph;
		\item \textbf{Algorithm~\ref{definition-farey-summation-algorithm}:} for a ray emanating from the origin we consider the union of simplices intersecting the ray. 
		This is the \textit{Farey polyhedron}, the central object of study, and was first studied by T.~Garrity in~\cite{Garrity2001};
		\item \textbf{Algorithm~\ref{definition-Meester-algorithm}:} the Meester algorithm, introduced by R.~W.~J~Meester in~\cite{Meester1989}, is the subtractive algorithm that defines continued fractions for the Farey polyhedron;
		\item \textbf{Algorithm~\ref{definition-Nose-stretching-algorithm}:} finally the \textit{nose stretching algorithm} provides a reconstruction of the Farey polyhedron from the continued fraction. 
	\end{itemize}
	We use the term \textit{polyhedron} for any dimension greater than~2.
	
	\vspace{2mm}
	
	\noindent
	{\bf Combinatorics of Farey polyhedra, prismatic diagrams, and sails.}
	We introduce a combinatorial description of the boundaries of Farey polyhedra using {\it prismatic diagrams}, special polyhedra with fixed triangulations.
	Prismatic diagrams are complete invariants of Farey polyhedra (see Theorem~\ref{theorem: prismatic diagrams invariant}). They  generalise the combinatorial description of triangulated polygons given by Farey boats, introduced by S.~Morier-Genoud and V.~Ovsienko in~\cite{moriergenoud2019}.
	Prismatic diagrams allow us to generalise several important notions of the geometry of continued fractions to the multidimensional case: sails and LLS sequences, whose elements encode the elements of the Farey summation continued fractions
	(see Definition~\ref{def:LLS-3D} and Theorem~\ref{3D-LLS-invariance}). Prismatic diagrams and their LLS sequences are suitable tools with which to study geometric aspects of the Farey summation algorithm.
	
	\vspace{1mm}

	\vspace{2mm} 
	
	\noindent
	{\bf Generalised frieze patterns.} 
	In 1973 J.~Conway and H.~S.~M.~Coxeter discovered a one-to-one correspondence between triangulated polygons and frieze patterns, where each pair of vertices corresponds to a unique frieze element.
	The defining relation of frieze patterns is then interpreted as a Ptolemy relation on the edges of the corresponding triangulated polygon.
	
	\vspace{1mm}
	
	A natural question to ask is as follows: {\it is there an analogous frieze-like Ptolemy relation for faces of the Farey polyhedron, described by invariants of the prismatic diagram?}
	We answer this question in the affirmative in Theorem~\ref{theorem: ptolemy in 3d}.
	We use newly defined three-dimensional continuants 
	to find $\lambda$\textit{-lengths}, values assigned to pairs of vertices of Farey polyhedra.
	We show that the $3\times3$ matrix of $\lambda$-lengths defined by pairs of faces of the Farey polyhedron have determinant $1$.

	\vspace{2mm}

	\noindent\textbf{Organisation of the paper.}
	We start in Section~\ref{Two-dimensional case} with a small survey of the classical two-dimensional case, collecting together all notions that we generalise to the multidimensional case later.

\vspace{1mm}
    
    We continue in Section~\ref{section: construction of all algorithms} with the discussion of Farey tessellation and the construction of the \textit{Farey polyhedra}, the central objects of study in this paper.
	
	\vspace{1mm}
	
	We recall some basic notions of integer geometry in Subsection~\ref{int geo subsection} before defining the Farey tessellation in the integer setting in Subsection~\ref{subsection: farey tessellation}.
	Farey polyhedra are defined by the Farey summation algorithm, described in Subsection~\ref{subsection: farey summation algorithm}.
	Here we also introduce important terminology used throughout the paper.
	A multidimensional continued fraction is then defined from the Farey summation algorithm in Subsection~\ref{subsection: farey summation cf}.
	
	\vspace{1mm}
	
	In Subsection~\ref{subsection: Meester algorithm} we recall the algorithm of R.~W.~J~Meester and note its equivalence to the continued fraction defined by the Farey summation algorithm.
	Finally we recall the situation of convergence of the Meester algorithm in Subsection~\ref{Convergence of the Farey summation algorithm}, and relate these known results to the Farey summation algorithm.
	
	\vspace{2mm}
	
	In Section~\ref{section: properties of farey polyhedra} we study the properties of Farey polyhedra.
	We introduce important invariants of the polyhedra, including the combinatoric prismatic diagrams.
	
	\vspace{1mm}
	We start with Subsection~\ref{subsection: basic properties tessellation}
	in which we discuss the basic properties of the Farey tessellation. 
	Prismatic diagrams are defined in Subsection~\ref{subsection: prismatic triangulation}.
	
	We define generalisations of the important integer invariants, sails and LLS sequences, from prismatic diagrams in Subsection~\ref{Sails}.
	
	\vspace{1mm}

	We study the matrix decomposition of the Farey summation algorithm in Subsection~\ref{subsection: Semi-group of matrices}.
	This is a prerequisite for the study of LLS sequences in Subsection~\ref{LLS sequence and the elements}.
	The matrix decomposition admits a natural generalisation of the notion of \textit{continuants}.
	We discuss this in Subsection~\ref{subsection: continuants}.

	\vspace{1mm}
	In Subsection~\ref{LLS sequence and the elements} we describe the Ptolemy relation for Farey polyhedra, and introduce three-dimensional frieze patterns.
	
	\vspace{2mm}
	
	Finally in Section~\ref{Plans for further work} we mention some open questions in the area.

\section{Two-dimensional case}
\label{Two-dimensional case}

Let us begin our study by collecting together several notions covering the two dimensional case.

\noindent
{\bf Notion A. Regular continued fractions.}
The expression 
$$
r=
a_0+\frac{1}
{\displaystyle a_1+\frac{1}
{\displaystyle a_2+\frac{1}{\displaystyle \ldots+\frac{1}{
\displaystyle a_n}}}}
$$
is called a {\it continued fraction} of $r$.
Here $a_i$ are the elements of the continued fraction. Denote it by $[a_0;a_1:\ldots: a_n]$

\vspace{2mm}

A continued fraction $[a_0;a_1:\ldots: a_n]$
is {\it regular} if $a_0$ is integer and all other elements are positive integers. Every rational number has exactly two such representations: one with an odd number of terms and one with an even number of terms.

\begin{example} For $r=7/5$ we have: 
$$
\frac{7}{5}=[1;2:2]=[1;2:1:1].    
$$
\end{example}

\noindent
{\bf Notion B. Slow Euclidean Algorithm.}

	\begin{mybox}[label=definition-slow-euclidean-algorithm]{Slow Euclidean Algorithm}
		The \textit{slow Euclidean algorithm} is the algorithm that produces a (finite or infinite) sequence of pairs of integers $(a_i,b_i)$ for a given pair $(a,b)$.
		
		\vspace{2mm}
		
		\noindent
		\textit{\underline{Input:}}
        We start with a pair two positive integers $(a,b)$, which we label as $(a_1,b_1)$. Without loss of generality we assume that $a_1\leq b_1$.

        \vspace{2mm}
        \noindent\textit{\underline{Step~$i$:}}
        In the $i$-th step of the algorithm we generate the pair $(a_{i+1},b_{i+1})$ by the formula:
        \[
        (a_{i+1},b_{i+1})
        =
        \left\{
        \begin{array}{l}
        (a_i-b_i,b_i) \quad \hbox{if $a_i\ge b_i$;}\\
        (a_i,b_i-a_i) \quad \hbox{if $a_i< b_i$.}
        \end{array}
        \right.
        \]
        The algorithm terminates when one of the elements equals zero.
        
		\vspace{2mm}
		
        \noindent\textit{\underline{Output:}} For the output we will collect all the pairs of integers  that were constructed at the steps of the algorithm (we also add the initial pair).

	\end{mybox}

\begin{remark}
Traditionally, the final output is the last non-zero number, which equals $\gcd(a,b)$. 
In Remark~\ref{remark: integer length is gcd} we consider this value as an invariant of \textit{affine integer transformations}.
\end{remark}

\begin{example}
The output of the slow Euclidean algorithm for $(5,7)$ is as follows:
$$
\{(5,7),(5,2),(3,2),(1,2),(1,1),(0,1)\}
$$
\end{example}

\vspace{2mm}

\noindent
{\bf Notion C. Sail of continued fractions.}
Let us continue with the following construction that provides the link from the theory of continued fractions to the geometry of numbers.

\vspace{2mm}

\begin{definition}
Assume that an angle $\angle ABC$ has an integer vertex $B$ and that the edges $AB$ and $BC$ are not collinear.
Then the boundary of all integer points in the interior of $\angle ABC$ together with its boundary but excluding $B$ is called the {\it sail} of the angle.  
\end{definition}

The sail of an angle is a broken line with infinitely or finitely many segments and zero, one, or two infinite rays. 
In what follows we are mostly interested in the union of segments, which is called the {\it the principal part of the sail}.

\begin{figure}[t]
$$
\includegraphics[width=.32\textwidth]{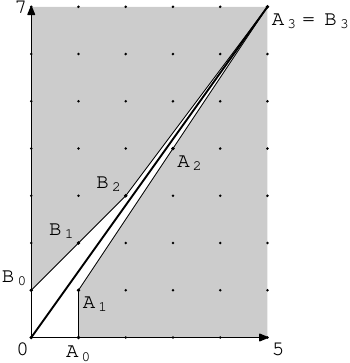}
\qquad\quad
\includegraphics[width=.32\textwidth]{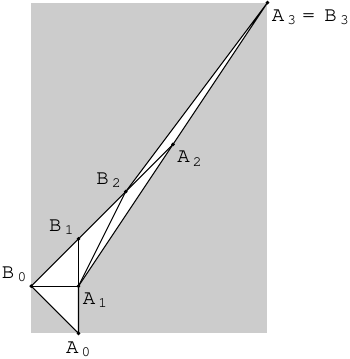}
$$
\caption{Principal parts of sails for 7/5 (left) and the corresponding Farey boat (right).}
\label{2d-intro-3-4.mp}
\end{figure}

\vspace{2mm}

Now we link sails to the geometry of numbers. 
First let us fix some notation.

\begin{definition}
Let $r \ge 1$  be a real number. 
Consider the ray $y=rx$ with vertex at the origin. 
This ray splits the positive quadrant into two angles
$\alpha_x(r)$ and $\alpha_y(r)$ containing the $x$-axis and the $y$-axis respectively.
Denote all the integer vertices of the principal parts of the sail of $\alpha_x(r)$ and $\alpha_y(r)$ by $A_0,\ldots, A_n$
and by $B_0,\ldots, B_m$ respectively,
as seen, for example, in Figure~\ref{2d-intro-3-4.mp}, left.
\end{definition}

\noindent
{\bf Notion D. Farey boats.}
Let $r$ be a positive rational number.
Consider the graph whose vertices are 
$A_i$ and $B_j$ of the corresponding sails for $r$. 
The edges are:
\begin{itemize}
\item $A_0B_0$;
\item $A_iA_{i+1}$ and $B_jB_{j+1}$ for all admissible $i$ and $j$; 

\item the edge $A_iB_j$ is taken if and only if $A_iB_j$ does not intersect the sails for $r$ at the points distinct to the vertices (for all admissible $i$ and $j$).
\end{itemize}
This graph is 
known as {\it Farey boat} for $r$, introduced in~\cite{moriergenoud2019}.
An example of the Farey boat embedded in $\Z^2$ is seen in Figure~\ref{2d-intro-3-4.mp}, right.
Let us equip every vertex $(a,b)$ of the Farey boat with the rational number $b/a$, as seen, for example, in Figure~\ref{2d-intro-1-2.mp}, right.

\vspace{2mm}

We have the following nice properties 
of Farey boats.

\begin{definition}
    \label{definition: farey sum}
    
	Consider two rational numbers $p/q$ and $r/s$ defined by the two pairs of relatively prime numbers $(p, q)$ and $(r,s)$.
	The \textit{Farey sum} of these two numbers is defined as follows:
	\[
	\frac{p}{q}\oplus\frac{r}{s}=\frac{p+r}{q+s}.
	\]
\end{definition}

\begin{proposition}
 Three numbers $a_1/b_1$, $a_2/b_2$, $a_3/b_3$ of the three consecutive vertices on the union of principal parts of the sails of any $v$ satisfy
\[
\frac{a_2}{b_2}=
\frac{k(a_1+a_3)}{k(b_1+b_3)}=
\frac{a_1}{b_1}\oplus\frac{a_3}{b_3},
\]   
where $k$ is the number of triangles in the Farey boat between the vertices $a_1/b_1$ and $a_3/b_3$. 
\end{proposition}

\begin{remark}
The combinatorics of the Farey boat are encoded by the slow Euclidean algorithm.
Namely we start with $A_0B_0$ and $(a,b)$. Every time when the first coordinate reduces we add the triangle $A_iA_{i+1}B_j$; while when the second coordinate reduces we add the triangle $A_iB_jB_{j+1}$.
\end{remark}

\begin{example}
As we have already seen, the
output of the slow Euclidean algorithm for $(5,7)$ is as follows:
$$
\{(5,7),(5,2),(3,2),(1,2),(1,1),(0,1)\}
$$
In the first step $(5,7)\to (5,2)$ the second coordinate reduces, so we start with the triangle $A_0A_1B_0$. 

Then, $(5,2)\to (3,2)\to (1,2)$. Here we add triangles $A_1B_0B_1$ and $A_1B_1B_2$.

Further $(1,2)\to (1,1)$; we add $A_1A_2B_2$.

In the last step $(1,1)\to (0,1)$ we can potentially reduce both coordinates to zero, so we add $A_2B_2R$, where $R=A_3=B_3$ by default.

We end up with the triangulation of Farey boat as shown on Figure~\ref{2d-intro-1-2.mp}, left.
\end{example}

\begin{figure}[t]
$$
\begin{array}{c}
\includegraphics[width=.375\textwidth]{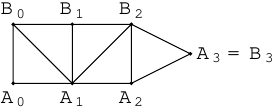}
\end{array}
\qquad\quad
\begin{array}{c}
\includegraphics[width=.27\textwidth]{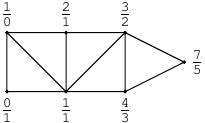}
\end{array}
$$
\caption{Farey boat with sail integer point labels (left) and rational values at vertices (right).}
\label{2d-intro-1-2.mp}
\end{figure}

\noindent
{\bf Notion E. Prismatic diagram.}
There is a natural way to represent Farey boats,
as a graph whose vertices are contained in two parallel lines.

\begin{definition}
We place $A_i$ at $(i-1, 0)$ and $B_j$ at $(j-1,1)$. Without loss of generality we set the last $A_m=B_n$ to be at $(\max(m,n)-1,1/2)$. We call such a representation the {\it prismatic diagram} of the Farey boat. 
\end{definition}

\vspace{2mm}

See Figure~\ref{2d-intro-1-2.mp} (left) for the example of $r=7/5$.

\vspace{2mm}

\noindent
{\bf Notion F. $\lambda$-length.}
We have the following definition.
\begin{definition}
\label{definition lambda length}
Let $a/b$ and $c/d$ be two non-negative irreducible rational fractions. Then the {\it $\lambda$-length} between them is set to be
\[
\det 
\begin{pmatrix}
a&c\\
b&d\\
\end{pmatrix}.
\]
\end{definition}
This definition is rooted in hyperbolic geometry, in the work of R.~Penner. See, for reference,~\cite{Penner2023}.

\vspace{2mm}

\noindent
{\bf Notion G. 
Ptolemy relation.}
Consider two arbitrary vertices of the Farey boat $v_1$ and $w_1$.
Say $v_2$ and $w_2$ are the adjacent vertices of $v_1$ and $w_1$ respectively such that  $(v_2,v_1,w_2,w_1)$ is a clockwise ordering of vertices of a quadrilateral in the Farey boat.
Then
$$
\det
\begin{pmatrix}
\lambda(v_1,w_1)&\lambda(v_1,w_2)\\
\lambda(v_2,w_1)&\lambda(v_2,w_2)\\
\end{pmatrix}
=\lambda(v_2,v_1)\lambda(w_1,w_2)=1.
$$

\begin{example}
Consider the Farey boat for $7/5$.
Let $v_1$ and $w_1$ be the vertices corresponding to $2/1$ and $1/1$ respectively.
Then $v_2$ and $w_2$ correspond to $1/0$ and $4/3$ respectively. 
$$
\includegraphics[width=.32\textwidth]{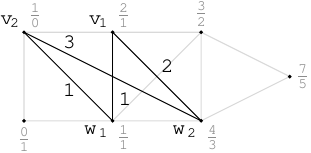}
$$
Then the Ptolemy relation for $v_1$, $v_2$, $w_1$, $w_2$ is as follows:
$$
\det
\begin{pmatrix}
\lambda(v_1,w_1)&\lambda(v_1,w_2)\\
\lambda(v_2,w_1)&\lambda(v_2,w_2)\\
\end{pmatrix}
=\det
\begin{pmatrix}
1&2\\
1&3\\
\end{pmatrix}
=1.
$$
\end{example}

\noindent
{\bf Notion H. Nose stretching algorithm.}
We present the classical algorithm for embedding the Farey boat into the positive quadrant.
Here we consider the labels $A_i$ and $B_j$ corresponding to fractions $a/b$, and so we add them with the Farey sum $A_i\oplus B_j$.
We could alternatively consider them as vectors once they are embedded in $\Z^2$, in which case we could write $A_i+B_j$, where this is vector addition.

	\begin{mybox}[label=definition-nose-stretching-algorithm-2d]{Nose Stretching Algorithm}
		The \textit{nose stretching algorithm} is the algorithm that produces an embedding of a Farey boat in the positive quadrant.
		
		\vspace{2mm}
		
		\noindent
		\textit{\underline{Input:}}
        From the Farey boat we take the final vertex $R$ and the vertices of the vertices of sails $A_i$ and $B_j$.
        
        \vspace{2mm}
        \noindent\textit{\underline{Step~$k$:}}
        In the $k$-th step of the algorithm we have a polyhedron in the positive quadrant and a ray that passes through two external edges, $A_0B_0$ and $A_iB_j$.
        (If $k=1$ then we have have one external edge, and we construct $C=A_0\oplus B_0$.)
        We construct the vertex $C=A_i\oplus B_j$ and the edges $A_iC$ and $B_jC$.
        If the ray $OR$ passes through the edge $A_iC$ we relabel $C=A_{i+1}$, otherwise we relabel as $C=B_{j+1}$.
        
        The algorithm terminates when $A_i\oplus B_j=R$.
        
		\vspace{2mm}
		
        \noindent\textit{\underline{Output:}} We have the embedding of the Farey boat in the positive quadrant.
		
	\end{mybox}

\begin{figure}
$$
\begin{array}{c}
\includegraphics[width=.3\textwidth]{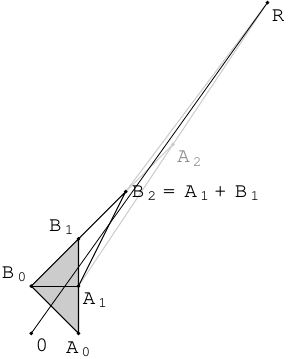}
\end{array}
$$
\caption{A step of nose's stretching algorithm.}
\label{2d-intro-1-6.mp}
\end{figure}

\begin{example}
In Figure~\ref{2d-intro-1-6.mp} we consider $R=(5,7)$. 
Assume that we have constructed $A_1B_1$ and everything to the left of it (namely $A_0$, $A_1$, $B_0$, $B_1$, and two grey triangles they form). 
Next we take $A_1\oplus B_1$ and set is as $B_2$ since the ray $OR$ intersects the segment $B_2A_1$.
\end{example}

\begin{remark}
The nose stretching algorithm is essential for building the Farey boat because it avoids the need to construct convex hulls. We use this later to generalise the notion of Farey boats to the the multidimensional case.
\end{remark}

\vspace{2mm}

\noindent
{\bf Notion I. LLS-sequences.} Now we recall the notion of lattice length-sine sequences, which serves as a complete invariant of lattice angles.

\vspace{1mm}

Let us begin with the following three integer invariants:
\begin{itemize}
\item 
{
{\it Integer length} of an integer segment $AB$ is the number of integer points in the interior of $AB$ plus 1.
Denote it by $\il(AB)$.
}
\item 
{
{\it Integer area} of a triangle $ABC$
is the index of sublattice generated by $AB$ and $AC$ in the lattice $\z^2$. Denote it by $\iv(ABC)$.
}
\item 
{
{\it Integer sine} of an angle $\angle ABC$
is set as 
$$
\isin\angle ABC =\frac{\iv(ABC)}{\il (AB)\il{BC}}.
$$
}
\end{itemize}

\begin{remark}
\label{remark: integer length is gcd}
    Note that the integer distance between two points with integer coordinates $(a_1,a_2)$ and $(b_1,b_n)$ is the value
    \[
    \gcd(a_1-b_1,a_2-b_2).
    \]
\end{remark}

\begin{definition}
Let $r\geq1$ be rational.
Let the integer vertices of the principal parts of the sail of $\alpha_x(r)$ be $V_0,V_1,\ldots,V_m$ (here we start with $V_0=(1,0)$).
Then the \textit{LLS-sequence} is defined to be
$$
\Big(\il(V_0V_1),\,\isin(V_0V_1V_2),\,\il(V_1,V_2),\ldots, \isin(V_{m-2}V_{m-1}V_{m}),\,
\il(V_{m-1}V_m)\Big).
$$
\end{definition}

\begin{figure}[t]
$$
\includegraphics[width=.32\textwidth]{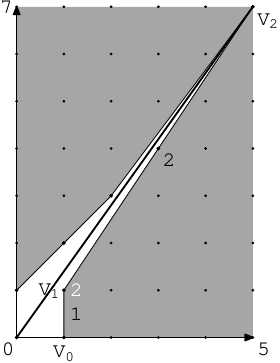}
\qquad\quad
\includegraphics[width=.32\textwidth]{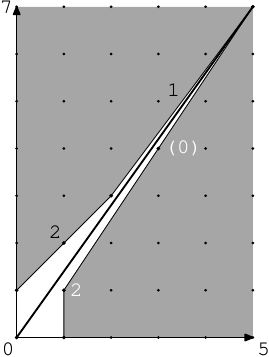}
$$
\caption{LLS-sequence for 7/5 (left) and the integer length-sine duality (right).}
\label{2d-intro-7-8.mp}
\end{figure}

\begin{example}
On Figure~\ref{2d-intro-7-8.mp} (left) we show the LLS-sequence for $7/5$. It is $(1,2,2)$. The integer lengths are shown in black while the integer sine is in white. Note that
$$
\frac{7}{5}=[2;\,2:1].
$$
Here the elements of the odd regular continued fraction coincides with the LLS. This is the case for all rational numbers $r\ge 1$.
\end{example}

\vspace{2mm}

\noindent
{\bf Notion J. Integer length-sine duality.} The duality notion is summarised in the following proposition.

\begin{proposition}
Consider a rational $r\geq1$.
Then the sequence of integer lengths for the sail of $\alpha_y(r)$ coincides with the sequence of integer sines for the sail of $\alpha_x(r)$ with the only exception in the last term, where the value is reduced by 1.
\end{proposition}

\begin{example}
On Figure~\ref{2d-intro-7-8.mp} (right) we show the 
integer length-sine duality. Here the integer lengths of the sail $\alpha_y(7/5)$ consists of two segments of length $2$ and $1$  respectively. 
The first integer length corresponds to the first integer sine of $\alpha_x(r)$. The last one equals $1$ and hence it corresponds to $1-1=0$ which is not actually present in the LLS-sequence (as it is zero). It can be formally considered as the straight angle at the middle vertex of the last integer segment  of the sail for $\alpha_x(r)$.

\end{example}

\begin{remark}
    In fact, the first construction of the LLS-sequence is due to F.~Klein, who constructed it solely from integer lengths in~\cite{Klein1895}.
\end{remark}

To simplify paper navigation we collect the references to the multidimensional analogues in the following table.

\vspace{2mm}

\begin{center}
\begin{tabular}{|c|c|c|} 
\hline
\textbf{Notion}& 
\textbf{Name} &
\textbf{Extension reference} \\
\hline
\hline
A & Regular continued fractions & Meester algorithm II: tabulation in Subsection~\ref{A general definition}\\
\hline
B & Slow Euclidean algorithm & 
Meester algorithm (Algorithm~\ref{definition-Meester-algorithm} on page~\pageref{definition-Meester-algorithm})\\
\hline
C & Sails of continued fractions & 
Definition~\ref{3d_saiL} on page~\pageref{3d_saiL}\\
\hline
D & Farey boats & Definition~\ref{FareyPolyhedron} on page~\pageref{FareyPolyhedron} \\
\hline
E & Prismatic diagram & Definition~\ref{pathtriangulated} on page~\pageref{pathtriangulated} \\
\hline
F & $\lambda$-length & Definition~\ref{lambda3-length} on page~\pageref{lambda3-length} \\
\hline
G & Ptolemy relation & Theorem~\ref{lambda3-length} on page~\pageref{Ptolemy-3dd} \\
\hline
H & Nose stretching algorithm & Algorithm~\ref{definition-Nose-stretching-algorithm} on page~\pageref{definition-Nose-stretching-algorithm} and
Subsection~\ref{A rigid structure of Farey masts, nose stretching}\\
\hline
I & LLS-sequences &Definition~\ref{def:LLS-3D} on page~\pageref{def:LLS-3D} \\
\hline
J & length-sine duality & Proposition~\ref{On duality-prop} on page~\pageref{On duality-prop} \\
\hline
\end{tabular}
\end{center}

\vspace{1mm}

Notions $C$-$J$ are new in the multidimensional case.

	\section{Farey tessellation and related algorithms}
	\label{section: construction of all algorithms}
	
	We start this section with a short discussion on integer geometry, which forms the basis of much of the paper, followed by a description of the Farey tessellation.
	Then we define the Farey summation and the Meester algorithms which are closely linked together, and they are both related to the geometric continued fraction that we study.
	We close the section with some short words on the convergence of these algorithms.

	\subsection{A few words on integer geometry}
	\label{int geo subsection}

	We work in affine integer geometry with points and vectors as objects: two points $A$ and $B$ may generate a vector $AB=B-A$.

	Let us fix some integer $n\ge 2$.
	A point in $\r^n$ is said to be {\it integer} if its coordinates are integer. All integer points form the lattice of integer points $\z^n$.
	A vector is \textit{integer} if all of its coordinates are integer.
	A segment, polygon, or polyhedron are {\it integer} if all their vertices are integer.

	\vspace{2mm}
	
	We say that an affine transformation is {\it integer} if it preserves the lattice of integer points. The set of affine transformation is denoted by $\aff(n,\z)$. Any integer affine transformation is a composition of a $\GL(n,\z)$ matrix multiplication and a shift on an integer vector.

	\begin{definition}
		Two sets $S_1$ and $S_2$ are {\it integer congruent} if there is a bijective integer affine transformation between $S_1$ and $S_2$.
	\end{definition}
	
	\begin{definition}
		{\it Integer length} of an integer segment is the number of integer points in the interior if this segment plus one. The {\it integer distance} between two integer points is the integer length of the segment connecting them. 
	\end{definition}

	We define integer volume in terms of index of sublattices. 
	We refer the reader to~\cite[Section~$2.1.3$]{karpenkov-book-v2} for more details.
	\begin{definition}
		The \textit{integer volume} of a simplex in $\Z^n$ generated by linearly independent integer vectors $V=(V_i)_{i=1}^n$ is the index of the sublattice generated by $V$ in the integer lattice.
		Denote it by $\iv(V)$.    
	\end{definition}

	We will work often with points whose coordinates are relatively prime.
	We introduce the notion of a unit integer circle to describe such points.
	
	\begin{definition}
		Let $O\in\Z^n$.
		The {\it unit integer sphere $\z S^{n-1}\subset \z^n$ centred at $O$} is the set of all points whose integer distance to $O$ equals $1$.
		We call the unit integer sphere centred at $O$, $\Z S^1$, the \textit{unit integer circle centred at $O$}.
		Where appropriate we reduce the terminology to \textit{unit integer sphere $\z S^{n-1}$}.
	\end{definition}
	
	The left hand side of Figure~\ref{figure: unit circle and polygon} shows the points on the unit integer circle within the box centred at $O$ of side length~$17$.

	\begin{figure}[t]
		\[
		\includegraphics[width=.4\textwidth]{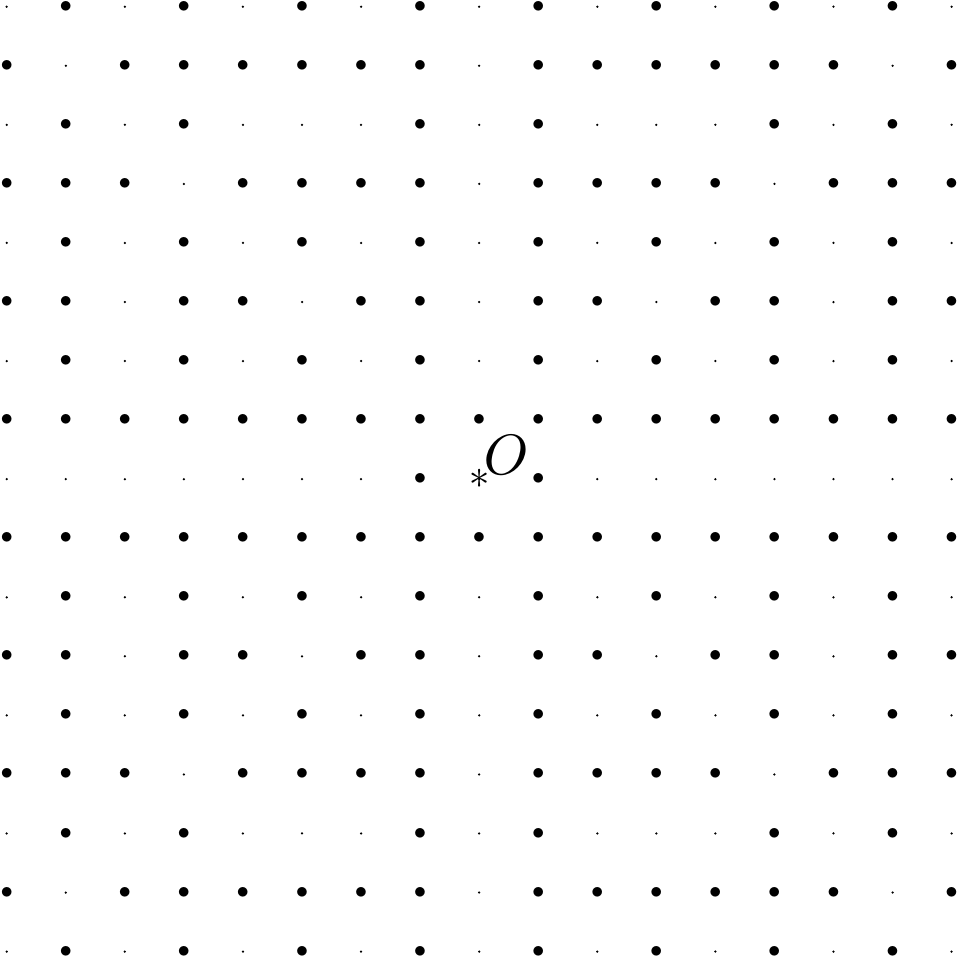}
		\qquad
		\qquad
		\includegraphics[width=.4\textwidth]{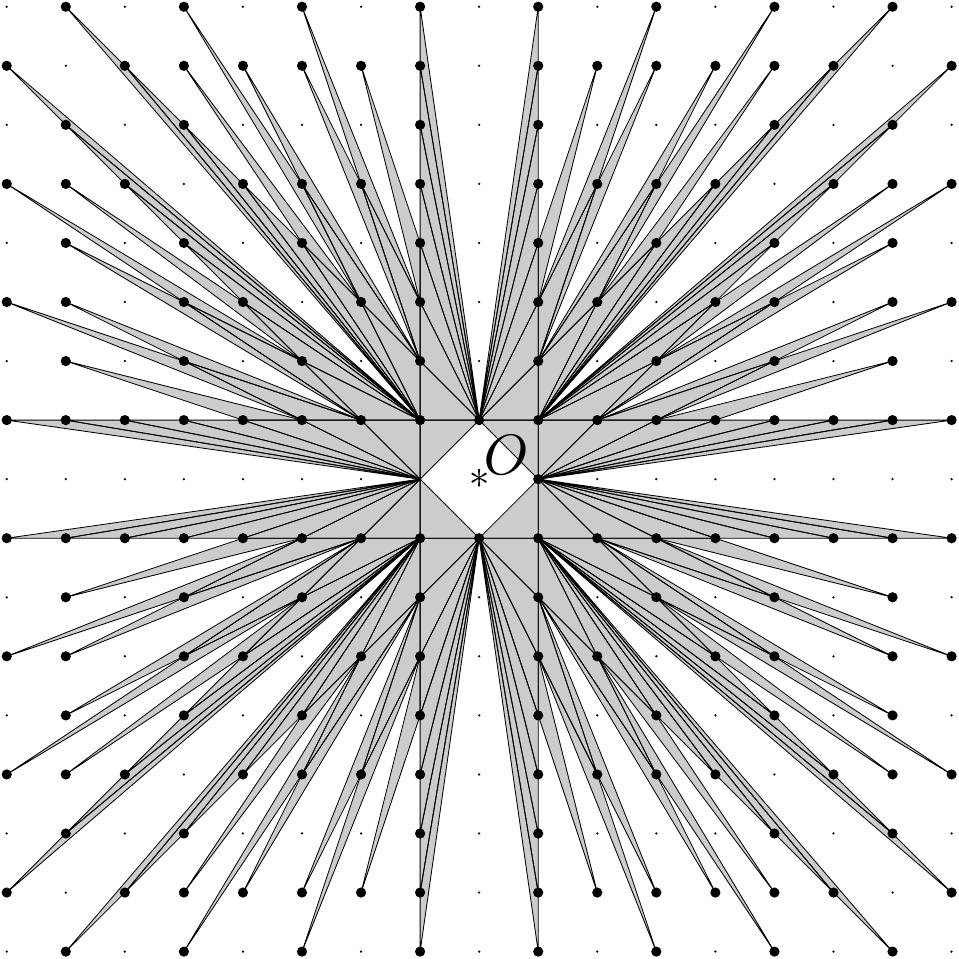}
		\]
		\caption{Unit integer circle (left) and Farey tessellation of the entire lattice (right).}
		\label{figure: unit circle and polygon}
	\end{figure}

	\subsection{Farey tessellation in higher-dimensions}
	\label{subsection: farey tessellation}
	Let us describe the Farey tessellation, recalling first the familiar two-dimensional case.
	
	\subsubsection{Construction of the Farey tessellation}
	
	Recall the Farey sum of Definition~\ref{definition: farey sum}.
	Consider an angle $\angle E_1 O E_2$ 
	and the corresponding coordinate system with basis $\{OE_1,OE_2\}$.
	Let $A=O+(p,q)$, $B=O+(r,s)$, and $C=O+(p+r,q+s)$ in this coordinate system ($O$, $A$, $B$, and $C$ are points, while $(p,q)$ and $(r,s)$ are vectors).
	Then the \textit{Farey summation of $A$ and $B$ is} 
	\[
	A\oplus_O B=C.
	\]
	Let us extend this definition to the multidimensional case in the following way.

	\begin{definition}
		Consider a cone $OE_1,\ldots, OE_n$ 
		and the  corresponding coordinate system.
		Let 
		\[
		A=O+(a_1,\ldots,a_n),\qquad B=O+(b_1,\ldots,b_n),\qquad C=O+(a_1+b_1,\ldots,a_n+b_n),
		\] 
		in this coordinate system.
		Then the \textit{Farey summation of $A$ and $B$ is} 
		\[
		A\oplus_O B=C.
		\]
		This summation is dependent on $O$; however, in this paper we fix $O$, and so we drop the subscript and simply write $A\oplus B$.

	\end{definition}

	\begin{remark}
		\label{remark: lattice volume}
		
		We denote a polyhedron with vertices $V_1,\ldots,V_k$ by $V_1\ldots V_k$. 
		Consider the dimension $k-1$ simplex $E_1\ldots E_k$ made up of standard basis vectors $E_i$.
		Let $E=E_1\oplus\ldots\oplus E_k$.
		From~\cite[Prop.~$18.11$]{karpenkov-book-v2} the dimension $k$ simplex $OE_1\ldots E_k$ has integer volume~$1$.
		To find the integer volume of the simplex $EE_1\ldots E_k$ one must take the greatest common divisor of the Pl\"ucker coordinates for the lattice generated by the vectors $EE_i$ for all $i=1,\ldots,k$, see~\cite[Thm.~$18.30$]{karpenkov-book-v2}.
		The integer volume of $EE_1\ldots E_k$ is $k-1$.
	\end{remark}

	Now we are ready to give the major building block of the multidimensional Farey tessellation that we describe later in Algorithm~\ref{FareyTess-def}.
	
	\begin{definition}
		Let $V_1\ldots V_k$ be a simplex in $\z S^n$ with dimension $k< n+1$. We say that its {\it Farey pyramid} is the
		($k{+}1$)-dimensional simplex $WV_1\ldots V_k$, where $W=V_1\oplus \cdots\oplus V_k$.
		Here $W$ is the vertex of the pyramid and $V_1\ldots V_k$ is its base.
		\\
		We say that a face of the Farey pyramid is {\it side/base} if it contains/does not contain $W$.
		\\
		In the degenerate case $k=1$ the vertex $W$ coincides with the base $V_1$.
		In this case we say that such pyramid does not have side faces.
	\end{definition}
	
	\begin{remark}
		Side faces of the Farey pyramid can have any dimension from $0$ to $k-1$. 
	\end{remark}
	
	Technically we do not require $W$ to be in~$\z S^{n-1}$, however in the below construction it will be always the case.
	
	\begin{mybox}[label=FareyTess-def]{Farey tessellation}
		\noindent\textbf{\underline{Input}:}
		Consider points $O$, $E_1,\ldots,E_n$ such that the vectors $OE_1,\ldots,OE_n$ form a lattice basis.
		Consider the positive orthant in this basis.
		
		\vspace{1mm}
		
		\noindent\textbf{\underline{Base of construction}:}
		Consider the set $S$ of all faces (of all possible dimensions) of the $(n-1)$-dimensional simplex $E_1 \ldots E_n$.  
		
		\vspace{1mm}
		
		\noindent\textbf{\underline{Step of construction}:}
		Let $S$ be the set of faces from the previous step. 
		At this step for every face in $S$ we add all side faces of its Farey pyramid to $S$.
		
		\vspace{2mm}
		
		We iterate the step of construction finite or infinitely many times.
		
		\begin{itemize}
			\item 
			the Farey pyramids of the resulting set $S$ are said to be {\it Farey simplices};

			\item
			the decomposition of the positive orthant into Farey simplices is said to be {\it Farey tessellation} of the positive orthant.
			(For completeness, we add the basis simplex to the tessellation.)

		\end{itemize}
	\end{mybox}

	\begin{remark}
		Later in Proposition~\ref{basic-properties-FT-extended}
		we show that the closure of the union of all the Farey simplices is the whole non-negative orthant and that two Farey simplices do not intersect.
	\end{remark}

	We may reflect the Farey tessellation around all coordinate axes.
	Such a tessellation is shown on the right hand side of Figure~{\ref{figure: unit circle and polygon}}, which shows the tessellation centred at the origin, within a box of edge length~$16$.
	
	\begin{remark} \label{remark: Farey sum preserves bases}
		Let $E$ be the Farey sum of the standard basis vectors $E=E_1\oplus \ldots \oplus E_n$.
		Note that the collections $E_1,\ldots,E_n$ and $E_1,\ldots,E_{i-1},E,E_{i+1},\ldots,E_n$, for any $i$, are related by a $\g_n(\Z)$ transformation.
		Hence, by induction, the vertices of the base of any Farey simplex is a basis for the $n$-dimensional integer lattice.
	\end{remark}
	
	\begin{remark}
		The central projection of all edges of the Farey tessellation to the plane $x_n=1$, namely
		\[
		(x_1,\ldots ,x_n) \to \Big(\frac{x_1}{x_n}, \ldots, \frac{x_{n-1}}{x_n} \Big),
		\]
		provides a triangulation of the basis simplex given by {\it triangle sequences} of~\cite{Garrity2004}. This triangulation is known as the {\it Farey partition} of the simplex.
	\end{remark}
	
	Let us mention the following quantitative properties of the Farey tessellation.
	
	\begin{proposition}\label{basic-properties-FT}
		\label{int geo prop}
		The following holds:
		
		\begin{itemize}
			\item The integer volume of a Farey simplex of dimension $k$ equals~$k{-}1$ $($with the only exception for the basis one whose integer volume is $1$$)$;

			\item The pyramid whose vertex is the origin and whose base is a face of any Farey simplex
			has a unit integer volume.
		\end{itemize}
		
	\end{proposition}

	\begin{proof}[Proof of Proposition~\ref{basic-properties-FT}]
		Let $W V_1\ldots V_k$ be a Farey simplex.
		Consider the integer sublattice whose basis is $OV_1, \ldots, OV_k$.		
		In this basis 
		\[
		W=V_1\oplus \cdots \oplus V_k=(1,\ldots,1,0,\ldots,0)
		\] 
		(here the first $k$ coordinates equal $1$).
		By Remark~\ref{remark: Farey sum preserves bases} the integer sublattice spanned by this basis is the integer lattice itself.
		Hence by Remark~\ref{remark: lattice volume} the integer volume of each Farey simplex of dimension $k$ is $k-1$, and any pyramids whose base is a face of a Farey simplex whose vertex is the origin has integer volume~$1$.
		
	\end{proof}
	
	The following corollary is necessary for the proof of Proposition~{\ref{basic-properties-FT-extended}} below.
	
	\begin{corollary} \label{cor: farey simplex = basis}
		The vertices of any face (of codimension~$1$) of a Farey simplex form a basis of the integer lattice.
		\qed
	\end{corollary}

	\subsubsection{Farey and quasi-Farey nets}
	
	Recall the following general definition.
	
	\begin{definition}
		We say that an integer simplex of dimension $n$ in $\z^n$ is {\it unimodular} if the vectors of its edges generate the lattice $\z^n$.
	\end{definition}
	
	Consider a tessellation $T$ of the positive orthant by polyhedra. Let us centrally project it to the coordinate triangle of the plane $x_n=1$. The image of the edges of the tessellation is called the {\it net} of this tessellation.
	
	\begin{definition}
		A net is said to be {\it Farey} if all the simplices in the original tessellation are unimodular. 
	\end{definition}
	
	Note that Farey nets were introduced by A. Hurwitz in~\cite{Hur-1894} and further developed by D.~J.~Grabiner in~\cite{Grab-1992}.   
	
	\vspace{2mm}
	
	In what follows we consider a slightly extended notion of Farey nets.
	
	\vspace{2mm}
	We say that a convex polyhedron with integer vertices is {\it empty} if it does not contain integer points distinct to its vertices (both in the interior and at the boundary).
	
	\begin{definition}
		A net is said to be {\it quasi-Farey} if all the simplices in the original tessellation are empty. 
	\end{definition}
	
	\begin{remark}
		\label{int geo remark}
		In the two-dimensional case all empty triangles are integer congruent to the coordinate one.
		However in higher-dimensions there are infinitely many different types of empty simplices that are integer non-congruent to each other. 
		All empty tetrahedra in dimension three have been classified by G.~K.~White in~\cite{White-1964}. The complete description of empty simplices in dimension greater than 3 is not known.
	\end{remark}

	\begin{proposition}
		The Farey tessellation of the coordinate cone generated by the standard basis vectors generates a quasi-Farey net.
	\end{proposition}

	\begin{proof}
		Recall that any simplex with integer volume~$1$ must be empty, by~\cite[Prop.~$18.11$]{karpenkov-book-v2}.
		
		Consider any Farey simplex $T$.
		By Proposition~\ref{basic-properties-FT}, the volumes of the pyramids with centre at the origin and bases at faces of $T$ equal $1$. 
		Therefore, they are all empty, and hence $T$ is empty as well.  
	\end{proof}
	
	\begin{remark}
		In the following section we introduce the \textit{Farey summation algorithm}, whose quasi-Farey net was studied by O.~R.~Beaver and  T.~Garrity in~\cite{Garrity2004} in dimension $3$. 
		The corresponding continued fraction map on the net 
		is called {\it the Triangle map}.
	\end{remark}

	\subsection{Description of the Farey summation algorithm}
	\label{subsection: farey summation algorithm}
	
	Let us describe a natural continued fraction algorithm generated by the Farey tessellation.

	\begin{mybox}[label=definition-farey-summation-algorithm]{Farey summation algorithm}
		The {\it Farey summation algorithm} is the algorithm that produces a (finite or infinite) sequence of simplices $T_i$ for a given vector $v\in \r^n_+$. 
		
		\vspace{2mm}
		
		\noindent
		{\it \underline{Base (Step~$0$) of the algorithm:}}
		Let $S_0$ be the basis
		simplex $OE_1\ldots E_n$.
		We formally set the simplex $E_1\ldots E_n$
		to be the \textit{deck $($zero yard$)$ $T_0$} of the algorithm.
		Here a \textit{yard} is an $(n-1)$-dimension simplex through which the ray in the direction of $v$ passes.
		
		\vspace{2mm}
		
		\noindent
		{\it \underline{Step~$i$ of the algorithm:}}
		In each Step~$k$ we construct the $k$-th yard $T_{k}$. 
		Hence we begin Step~$i$ with the $(i-1)$-th yard $T_{i-1}$ as input.
		
		Denote by $S_i$ the Farey pyramid with base $T_{i-1}$, we call it the {\it $i$-th Farey pyramid} of the Farey summation algorithm.
		Let us show how to construct the {\it $i$-th yard} $T_i$.
		Note that the ray in the direction of $v$ intersects exactly one side face of $S_i$ of any dimension in its interior (here interiors of zero-dimensional faces is the vertex itself). 
		We set $T_i$ to be this face.
		
		If the dimension of $T_i$ is lower than the dimension of $T_{i-1}$ we say \textit{the dimension drops at Step~$i$}.
		We demonstrate this step diagrammatically in Figure~{\ref{figure: farey summation step}}.
		
		\vspace{2mm}
		
		\noindent
		{\it \underline{Termination of the algorithm:}}
		In case the next chosen simplex is a single  vertex, the algorithm naturally terminates. 
		Otherwise, the algorithm produces an infinite sequence of yards.
		
		The last yard is called \textit{the $($crow's$)$ nest};
		the last vertex is called the 
		\textit{pennant}.
	\end{mybox}

	\begin{figure}[t]
		\[
		\includegraphics[width=.22\textwidth]{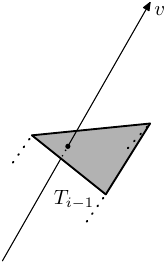}
		\qquad\quad
		\includegraphics[width=.22\textwidth]{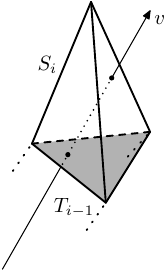}
		\qquad\quad
		\includegraphics[width=.22\textwidth]{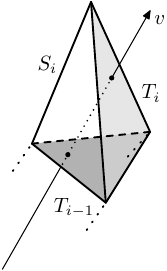}
		\]
		\caption{The $i$-th step in the Farey summation algorithm.}
		\label{figure: farey summation step}
	\end{figure}

	The next definition is a natural generalisation of Notion~D (Farey boats) from Section~\ref{Two-dimensional case}.    
	\begin{definition}\label{FareyPolyhedron}
		
		The union of the Farey pyramids formed by the algorithm for a given vector $v$ is called the \textit{Farey polyhedron of $v$}.
	\end{definition}
	
	The Farey polyhedron for the vector $(5,7,8)$ is shown in Figure~{\ref{sail3d-1-2.pdf}} below.
	The Farey summation algorithm may be tabulated as a continued fraction, and this is shown precisely in the following subsection.

	\begin{remark}
		Let $T_k$ and $T_l$ ($k<l$) be two yards for some vector $v$.
		Then the dimension of $T_l$ does not exceed the dimension of $T_k$.
		(The same for $S_k$ and $S_l$.)
	\end{remark}
	
	\begin{remark}
		Historically a multidimensional continued fraction algorithm is called {\it Farey} if its net is Farey. 
		However, while the multidimensional Farey tessellation of the Farey summation algorithm provides the most straightforward generalisation of the Farey tessellation of the plane, 
		its corresponding projection is a quasi-Farey net (and not a Farey net).
		For this reason the Farey summation algorithm is not really a Farey algorithm (however it provides the direct generalisation of the Farey addition). 
	\end{remark}

	\begin{remark}
		\label{remark: nautical terminology}
		The nautical terminology is motivated by two previous unrelated definitions: the \textit{sails} first studied by F.~Klein and so named by V.~I.~Arnold, and \textit{Farey boats} introduced by S.~Morier-Genoud and V.~Ovsienko, which we study further in Subsection~\ref{subsection: prismatic triangulation}. 
		
	\end{remark}

	\subsection{Definition of Farey summation continued fractions}
	\label{subsection: farey summation cf}
	
	We describe an algorithm for \textit{Farey summation continued fractions}, based on the Farey summation algorithm.
	For simplicity we work essentially in the three-dimensional case.
	
	\subsubsection{Enumeration of vertices}
	
	At each step of the Farey summation algorithm we have a Farey simplex $P_1$ as input and a second Farey simplex $P_2$ as output, connected by a yard.
	If $\dim(P_1)=\dim(P_2)$ then there is one vertex $v_1$ of $P_1$ that is not present in $P_2$, and one vertex $v_2$ of $P_2$ not present in $P_1$.
	In this way each step of the algorithm adds one vertex to the Farey polyhedron, and we say that $v_2$ \textit{replaces} $v_1$.

	Let us first enumerate the vertices in the algorithm.
	An \textit{external face} is any face of any dimension in the Farey polyhedron that is not a yard.
	Our base external face is $V_{0,1}=(1,0,\ldots,0)$, up to $V_{0,n}=(0,\ldots,0,1)$.  
	
	\vspace{2mm}
	
	\begin{mybox2}{Enumeration of Farey summation algorithm}
		We start the enumeration with the labelling of the basis vectors $V_{0,1},\ldots,V_{0,n}$ and the vector $v$.

		\vspace{2mm}
		
		\noindent
		As input for Step~$i{+}1$ we have an external face made up of $k\leq n$ vectors $V_{i,j_1},\ldots,V_{i,j_k}$, whose indices $(j_1,\ldots,j_k)$ form a subsequence of $(1,\ldots,n)$.
		
		We have two situations, either the dimension $k$ drops at Step~$i{+}1$ or it does not. 
		\begin{itemize}
			\item (No dimension drop): In this case one vertex is replaced by a new vertex. 
			Set the index of the new vertex to be the same as the replaced vertex (while the indices of the unchanged vectors stay the same). 
			For instance, if $V_{i,j_2}$ is replaced then we label: $V_{i+1,j_1}=V_{i,j_1}$; $V_{i+1,j_2}=V_{i,j_1}+\ldots+V_{i,j_k}$;  $V_{i+1,j_3}=V_{i,j_3}$; up to $V_{i+1,j_k}=V_{i,j_k}$.

			\item (Dimension drop): In this case the ray in the direction of $v$ passes through a $d$-face of dimension $d<k$ of the external face.
			We cease the labelling of all vectors not making up this $d$-face. 
			The new vertex in this step may take on any remaining index.
			(The choice corresponds to the number of zeros after a dimension drop seen in the Meester continued fraction, see Subsection~{\ref{subsubsection: nose stretching}} below.)
			The remaining indices are labelled in a cyclic order.
			
		\end{itemize}
		
	\end{mybox2}

	\begin{definition}
		\label{definition: principal}
		
		We say that the $i$-th yard ($i\ge 1$) of a Farey polyhedron is \textit{principal} if either the dimension drops at Step $i{+}1$ or if the index of the vertex added to the Farey polyhedron in Step $i{+}1$ is different to the index of the vertex added in Step $i$.
		The deck and the crow's nest (hereafter called simply the \textit{nest}), are considered to be \textit{principal} by default.
		
	\end{definition}

	\begin{figure}[t]
		\[
		\includegraphics[width=6.5cm]{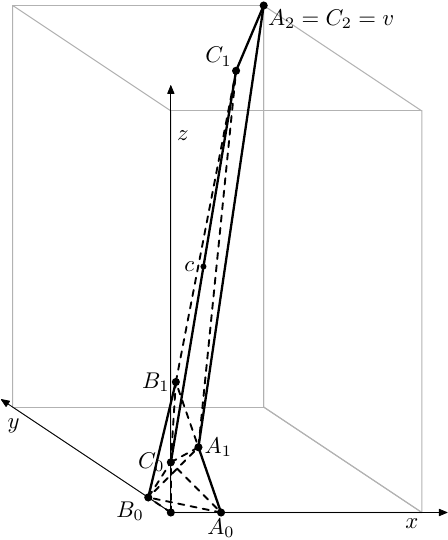}
		\qquad
		\qquad
		\includegraphics[width=7cm]{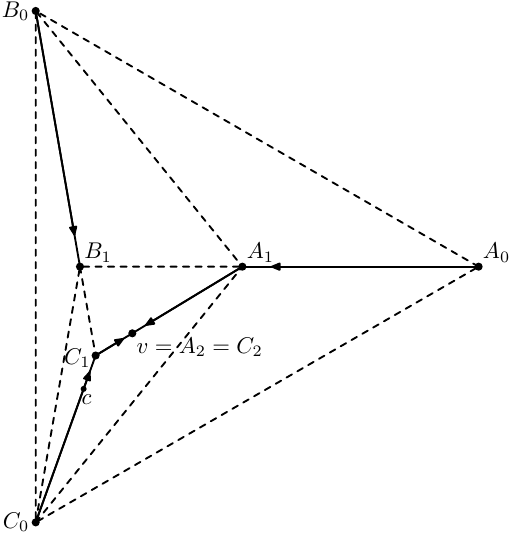}
		\]
		\caption{Farey summation algorithm for $(5,7,8)$ (left) and its central projection to the plane $x+y+z=1$.}
		\label{sail3d-1-2.pdf}
	\end{figure}
	
	Vertex enumeration allows us to track the supporting structure of the Farey polyhedra.
	To illustrate this, in the following example we relabel the standard basis from $\{E_0,E_1,E_2\}$ to $\{A_0,B_0,C_0\}$.

	\begin{example}\label{(5,7,8)-part 1}
		In Figure~\ref{sail3d-1-2.pdf} we show the Farey summation algorithm for the vector $(5,7,8)$.
		Note that in right hand picture's projection the origin is located at $A_1$.
		
		\begin{itemize}
			\item {\it Base:} We start with the basis tetrahedron $S_0=OA_0B_0C_0$ and set $T_0=A_0B_0C_0$, here $T_0$ is the deck.
			
			\item {\it Step~1:} We construct the Farey pyramid $S_1=A_0B_0C_0A_1$ with $A_1=A_0\oplus B_0 \oplus C_0 =(1,1,1)$. 
			The vector $v$ intersects the side $A_1B_0C_0$ (seen clearly from the right picture). 
			Hence $T_1=A_1B_0C_0$.
			
			\item {\it Step~2:} The next Farey pyramid is $A_1B_0C_0B_1$ with 
			$B_1=A_1\oplus B_0 \oplus C_0=(1,2,2)$.
			Hence $T_2=A_1B_1C_0$.
			
			\item {\it Step~3:}
			$S_3=A_1B_1C_0c$ with $c=(2,3,4)$; \\
			$T_3=A_1B_1c$. 
			The point $c$ is written in lowercase as $T_3$ is not a principal yard.
			
			\item {\it Step~4:}
			$S_4=A_1B_1c\, C_1$ with $C_1=(4,6,7)$;\\
			$T_4=A_1C_1$. Note that here the yard is one-dimensional. 
			
			\item {\it Step~5:}
			$S_5=A_1C_1v$. Since $T_4$ is one-dimensional, $S_5$ is two-dimensional.
			Here we arrive to the final vector 
			$v=A_1\oplus C_1=(5,7,8)$.\\
			The algorithm terminates here.
		\end{itemize}

		\vspace{2mm}
		
		In Figure~\ref{sail3d-1-2.pdf} (Left) we show all the Farey simplices for $(5,7,8)$ in the space; in Figure~\ref{sail3d-1-2.pdf} (Right) we show the central projection all the Farey simplices to the plane $x+y+z=1$. 
		As we will see later the point $v$ 
		can be labelled either as $A_2$ or $C_2$.
		
	\end{example}

	In the above example all the yards $T_i$ except for $i=3$ are principal; $T_0$ is a deck; $T_4$ is a nest; and the last vector $v$ is a pennant.

	\begin{example}
		
		\begin{figure}[t]
			\[
			\includegraphics[width=6.5cm]{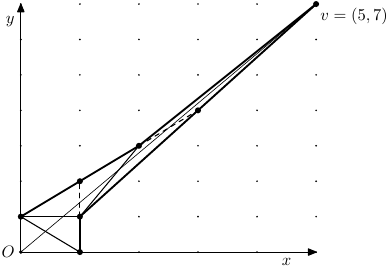}
			\]
			\caption{Farey summation algorithm for $(5,7)$.}
			\label{principalyardexample.pdf}
		\end{figure}
		
		Consider the Farey summation algorithm in two-dimensions for the vector $(5,7)$, shown in Figure~{\ref{principalyardexample.pdf}}.
		Note that the $(i-1)$-th and $i$-th principal yards (coloured grey in the figure) separate the Farey polygon into collections of $c_i$ triangles, where $c_i$ is the $i$-th element of the ordinary continued fraction of $7/5$, namely $[1;2:2]$.
		
	\end{example}

	Let us formulate the following important properties of Farey pyramids.
	
	\begin{proposition}
		Two consecutive principal yards of a Farey polyhedron
		share a face of codimension~1 $($with respect to the yard$)$.
		
		\vspace{2mm}
		
		\noindent
		The convex hull of two consecutive principal yards is a simplex; it coincides with the union of the Farey pyramids between them.
		\qed
	\end{proposition}
	
	As we will see later, while defining continued fractions we group consecutive Farey pyramids of the Farey summation algorithm together as follows.
	
	\begin{definition}
		\label{division simplex}
		The convex hull of two consecutive principal yards is said to be a \textit{division simplex} for the Farey polyhedron (division tetrahedron in three-dimension).  
	\end{definition}

	\subsubsection{Tabulation of the algorithm}\label{definition: Farey summation continued fractions}
	
	Let us present a tabulation of the Farey summation algorithm in three-dimensions.
	
	\begin{mybox2}{Farey summation algorithm continued fractions (Tabulation)}
		In the three-dimensional case the algorithm splits into two stages.

		\vspace{2mm}
		
		\noindent 
		{\it \underline{Stage~1 $($three-dimensional stage$)$:}} 
		
		\noindent Step~$i$ is in this stage if the $i$-th yard is two-dimensional, that is the simplices we generate are three-dimensional.
		
		\noindent Define $r_i\in \{1, 2,3\}$ to be the index of the vertex that is changing. 
		Let us write the sequence of $r_i$ generated in this stage in the following form: 
		\[
		1^{a_1}2^{a_2}3^{a_3}1^{a_4}\ldots,
		\] 
		where $a_j$ are the multiplicities of the consequent indices.

		\vspace{2mm}
		
		\noindent 
		{\it \underline{Stage~2 $($two-dimensional stage$)$:}}
		
		\noindent Step~$k{+}i$ is in this stage when the $(k+i)$-th yard is one-dimensional, that is the simplices we generate are two-dimensional.

		\noindent Here we change only the vertices with indices $s,t\in \{1,2,3\}$ such that $s\equiv k$ and $t\equiv k+1 \mod 3$.
		Again we can abbreviate the sequence of changed indices as:
		\[
		s^{b_1}t^{b_2}s^{b_3}t^{b_4}\ldots
		\]
		(note that $a_k$, $a_{k-1}$, and $b_1$ may be zero while passing from Stage $1$ to Stage $2$).
	\end{mybox2}

	\begin{remark}
		The proposed tabulation is very similar to LR-notation in the Farey graphs (see e.g. in~\cite{karpenkov-book-v2} for more details).
	\end{remark}
	
	Finally we join both of the sequences $(a_i)$ and $(b_j)$ together to form a single sequence.
	\begin{definition}
		\label{farey continued fractions}
		Consider a vector $v$ with positive coordinates. 
		Let it generate the sequences $(a_i)$ and $(b_j)$
		described above.
		If these sequences are finite, of length $k$ and $l$ respectively, the sequence
		\[
		[a_1;a_2:\cdots:a_k\, | \, b_1:\cdots : b_l]
		\]    
		is said to be the {\it Farey summation continued fraction}
		for a vector $v$.
		In case the sequences $(a_i)$ or $(b_i)$ are infinite we write
		\[
		[a_1;a_2:\cdots \, | ]
		\quad \hbox{and} \quad
		[a_1;a_2:\cdots:a_k\, | \, b_1:\cdots ]
		\]
		respectively.
	\end{definition}

	\begin{proposition}
		By construction all finite \emph{(}$i=1,\ldots,k$, $j=1,\ldots,l$\emph{)} and infinite sequences $(a_i)$ and  $(b_j)$ of non-negative
		elements are realised as Farey summation continued fractions for some vector $v$ with the following list of exceptions:
		\begin{itemize}
			\item There are no two consecutive zeros in $(a_i)$ except for:

			--- \, $a_1=a_2=0$ and $a_3\ne 0$;
			
			--- \, $a_{k-1}=a_k=0$, $a_{k-2}\ne 0$, and $(b_j)$ contains at least one positive element;
			
			\item If $(b_j)$ is empty, then $a_k\ne 0$.
			\item All $b_i$ are positive with the only exception that $b_1$ can be $0$ in the case $l\ge 2$.
		\end{itemize}
		\qed
	\end{proposition}
	
	\begin{example}\label{(5,7,8)-part 2}
		Consider a vector $v=(5,7,8)$. In Example~\ref{(5,7,8)-part 1} we have discussed the Farey summation algorithm for $v$. 
		Now we can write the Farey summation continued fraction for it:
		\[
		[1;1:2:0:0 \, |\, 1].
		\]
		Note that we add two zeros since our pair $(s{=}3, t{=}1)$ follows after $(s{=}1, t{=}2)$ and $(s{=}2,t{=}3)$ that we skip.
		\\
		Note also that the sum of all elements is $5$, which is equal to the number of steps (Farey pyramids) in the Farey summation algorithm.
	\end{example}

	\subsection{Meester algorithm}
	\label{subsection: Meester algorithm}
	
	Our next goal is to describe the 
	Meester algorithm which is a Jacobi-Perron type subtractive algorithm. 
	The Meester algorithm plays the same role for the Farey summation continued fractions
	as the Euclidean algorithm plays for continued fractions. In this subsection we work in an arbitrary dimension.

	\subsubsection{A general definition}
    \label{A general definition}
	Let us begin with a small remark.
	
	\begin{remark}
		In the Meester algorithm below we fix the order of basis vectors.
		A different ordering of basis vectors will lead to slightly different continued fractions but the same tessellation.
	\end{remark}

	We split the algorithm into two parts.
	The first part is subtractive.
	The second part is a tabulation of the data obtained.
    These two parts extend Notions~B and~A from Section~\ref{Two-dimensional case} to the multidimensional case respectively.
	
	\begin{mybox}[label=definition-Meester-algorithm]{Meester algorithm I: Subtraction}
		We start with an $n$-tuple of non-negative real numbers $(v_1,\ldots, v_n)$. 
		While at least one coordinate is positive we perform the following iteration steps.
		
		\vspace{2mm}
		
		The $i$-th iteration step is as follows
		($i=1,\ldots$): 
		\begin{itemize}
			\item Let $j\in\{1,\ldots, n\}$ satisfy  $j\equiv i \mod n$.
			
			\item 
			If the $j$-th coordinate is zero then we go straight to the next step.
			Otherwise we subtract the $j$-th coordinate from all other \underline{non-zero} coordinates simultaneously as many times as it is possible for all to remain  non-negative: denote the number of times we have subtracted by~$a_i$ (note that $a_i=0$ is possible).
			Denote also the resulting vector as 
			$(v_{i,1},\ldots, v_{i,n})$.

			\item 
			As a result of the previous item we obtain a new $n$-tuple of non-negative real numbers.
			We keep the number $a_i$ as 
			the element of the continued fraction.
		\end{itemize}
		
		The algorithm terminates if we have a single non-zero coordinate.

		\vspace{2mm}

	\end{mybox}
	
	\begin{example}
		\label{example: Meester algorithm}
		Consider the example of $(55,10,67)$. The Meester algorithm produces  
		\[
		(55,10,67) \overset{0,5}{\to} (5,10,17) \overset{0,2}{\to}
		(5,\underline{0},7) \overset{0,1}{\to} (5,0,2) \overset{2}{\to} (1,0,2)
		\overset{2}{\to} (1,0,\underline{0}).
		\]
		The zeros denote steps where no subtraction is possible.
		The sequence of $a_i$ here is $(0,5,0,2,0,1,2,2)$. 
	\end{example}
	
	\begin{mybox2}[label=definition-Meester-algorithm-tabulation]{Meester algorithm II: Tabulation}
		
		For any $j\in{1,\ldots, n}$ consider the sequence of a single coordinate $(v_{i,j})$. Denote by $s_j$ the integer for which $v_{s_j,j}=0$ and $v_{s_{j-1},j}>0$. If the sequence $v_{i,j}$ is always positive, then we do not assign any value for $s_j$.
		
		\noindent \textit{Note:} We introduce $s_j$ to keep track of the step at which the $j$-coordinate becomes $0$. 
		This is used later when reconstructing the Farey polyhedron from the continued fraction.

		\vspace{2mm}
		
		As the output of Meester algorithm we have two items:
		\begin{itemize}
			\item a sequence of non-negative integers $(a_i)$;
			\item 
			a sequence of $(s_{j_1},\ldots, s_{j_{k}})$ enumerated in increasing order. Here the number $k$ denotes the number of $s_{j}$ for which the values are assigned during the algorithm execution
			(here $k\le n-1$).
		\end{itemize}
		
		Finally we write the elements $a_i$ in the form of continued fraction.
		
		First, let us assume that the algorithm is finite, terminating after Step $N$.
		Then for the sequence of $s_{j_t}$ the range of $t$ is $1,\ldots,n-1$, as after step $N$ we have $n-1$ zero coordinates.
		We also do not indicate the last $s_{j_{n-1}}$ as this denotes a change occurring at the terminal step of the algorithm.
		The following expression is called the {\it Meester continued fraction}:
		\[
		\big[a_1;\cdots: a_{s_{j_1}} \,|_{j_1}\,
		a_{s_{j_1}+1}: \cdots :
		a_{s_{j_2}} \,|_{j_2}\,
		a_{s_{j_2}+1}:\cdots : 
		a_{s_{j_{n-2}}} \,|_{j_{n-2}
		}\,
		a_{s_{j_{n-2}+1}}: \cdots : a_N\big].
		\]
		
		In case the algorithm does not terminate, $k$ could be any number of  $\{1,\ldots, n-2\}$. The corresponding continued fraction is infinite:
		\[
		\big[a_1;\ldots: a_{s_{j_1}} \,|_{j_1}\,
		a_{s_{j_1}+1}: \cdots :
		a_{s_{j_2}} \,|_{j_2}\,
		a_{s_{j_2}+1}:\cdots : 
		a_{s_{j_{k}}} \,|_{j_{k}
		}\,
		a_{s_{j_{k}+1}}: a_{s_{j_{k}+2}}: \cdots \big].
		\]
		For both finite and infinite continued fractions: after the symbol $|_{j_s}$ is used we do not include $a_i$, where $i=j_s\mod n$, in the continued fraction, since the values of $a_i$ are not assigned at these steps.
		
	\end{mybox2}

	\begin{continueexample}{example: Meester algorithm}
		Further we have: 
		\begin{itemize}
			\item 
			$s_{j_1}=4$ and $j_1=2$, 
			from the third vector we have the zero second coordinate (underlined);
			\item $s_{j_2}=8$ and $j_2=3$ in the last vector the third coordinate is zero (underlined).  For simplicity we do not show 
			$|_{j_2=3}$ in the continued fraction, which stands in the last position (after which no steps are done).
		\end{itemize}
		
		\noindent
		Therefore, the Meester continued fraction for $(55,10,67)$ is as follows: 
		\[
		[0;\, 5: 0: 2 \, |_{j_1=2} \, 0:1:2:2].
		\]
	\end{continueexample}
	
	\begin{remark}
		\label{future work 1}
		The Meester algorithm has been previously introduced and studied in~\cite{KM1995,Meester1989,MN1989}.
	\end{remark}

	\begin{remark}
		It is a simple exercise to show that the value of the non-zero coordinate in the final step of the algorithm is the greatest common divisor of the initial coordinates, just as in the classical Euclidean algorithm.
	\end{remark}
	
	The following proposition links the geometrical nature of Farey summation algorithm and the number theoretical nature of Meester-algorithm.

	\begin{proposition}
		\label{prop: both cf the same}
		Let $v$ be a non-zero point of the positive orthant.
		Then the Meester algorithm for $v$ generates the same continued fraction as the Farey summation algorithm.
		
		The only difference in writing is as follows. One should replace every 
		\[
		|_j\quad \mbox{ with }\quad 0:\cdots:0:|,
		\]
		where the number of zeros is the number of cyclic transpositions to exclude the correct coordinate. In the three-dimensional case it is either $0$, $1$, or $2$. 
		\qed
	\end{proposition}
	
	\begin{remark}
		Here and below we may use the term \textit{continued fraction} to refer to Farey summation continued fractions and Meester continued fractions interchangeably, unless otherwise specified.
	\end{remark}

	\begin{example}\label{(5,7,8)-part 3}
		For the sequence of Example~\ref{(5,7,8)-part 2} we have
		\[
		(5,7,8) \to (5,2,3) \to (3,2,1) \to (1,0,1) \to (1,0,0).
		\]
		The corresponding Meester continued fractions is 
		\[
		[1;1:2\, |_2\, 1]=[1;1:2:0:0\, |\, 1].
		\]
		
	\end{example}
	
	\subsubsection{Extended Meester continued fraction}
	\label{subsubsection: nose stretching}

	Recall the relation between even and odd ordinary continued fractions:
	\[
	[a_0;\cdots:a_n]=[a_0; \cdots: a_{n}-1: 1]
	\]
	(here we assume that $a_n>1$).
	Note that the Euclidean algorithm does not generate $[a_0; \cdots: a_{n}-1: 1]$,
	however both forms find use in the literature.
	For example, in integer geometry one may be interested solely in odd length continued fractions.
	We consider analogous equivalent continued fractions for the Farey summation algorithm.

	\begin{definition}\label{extended-cf}
		Consider the Meester continued fraction.
		Let the last step be on $k$ non-zero coordinates.
		Let us replace the last element $a_n$ of the continued fraction by a sequence
		$(a_{n}-1,0,\ldots, 0,  1)$, where the number of zeros does not exceed $k-2$.
		We say that the obtained sequence is the {\it extended Meester continued fraction}.
	\end{definition}
	
	\begin{example}
		For $(5,7,8)$ we have the following two extended Meester continued fractions
		\[
		[1;1:2\,|_2 \, 1]=[1;1:2\,|_2 \,0:1].
		\]
	\end{example}

	\subsubsection{Reconstruction of Farey simplices for a given continued fraction}
	\label{definition: nose stretching}
	
	Let us show how to reconstruct all the Farey simplices that were used in the Farey summation algorithm from the Farey summation continued fraction.
    This generalises the two dimensional nose stretching algorithm of Notion~H, Algorithm~\ref{definition-nose-stretching-algorithm-2d}.
	
	\vspace{2mm}
	
	\begin{mybox}[label=definition-Nose-stretching-algorithm]{Nose stretching algorithm}
		We start with an $n$-tuple of integer basis vectors $(E_1,\ldots, E_n)$. We are also given a Farey summation continued fraction
		\[
		\alpha=\big[a_1;\cdots: a_{s_{j_1}} \,|_{j_1}\,
		a_{s_{j_1}+1}: \cdots :
		a_{s_{j_2}} \,|_{j_2}\,
		a_{s_{j_2}+1}:\cdots : 
		a_{s_{j_{n-2}}} \,|_{j_{n-2}
		}\,
		a_{s_{j_{n-2}+1}}: \cdots : a_N\big].
		\]
		
		\noindent
		{\bf Step of the algorithm.} On Step~$s$ we get several ordered vectors $(V_1,\ldots,V_n)$.
		The step contains two stages.
		
		\vspace{1mm}
		
		\noindent{\it \underline{ Generating vectors for the next step:}}
		We start this step with the vector $V_j$, the first non-zero vector to the right of $V_{s \mod n}$. 
		In other words, we skip any zero vectors.
		Let us use the interim label $V_j'=V_j$.
		Let $k$ be the number of symbols $|_i$ occurring in all previous steps.
		We denote
		\[
		\hat{V}=V_j'\oplus a_s(V_1\oplus\cdots\oplus V_{j-1}\oplus V_{j+1} \oplus\cdots\oplus V_n)
		\]
		(recall $a_s$ is the $s$-th element of the Farey summation continued fraction).
		We relabel $V_j=\hat{V}$, so the data for the next step will be the vectors
		\[
		V_1,\ldots,V_n.
		\]
		
		\vspace{1mm}
		
		\noindent{\it \underline{Vector erasing:}} 
		The vectors $V_i$ whose indices are given by symbols $|_i$ occurring on Step $s$ are all replaced with zero vectors (recall there may be multiple $|_j$ at Step~$s$).
		
		\vspace{1mm}
		
		\noindent{\it 
			\underline{Returning a Farey simplex $S_s$:}}
		On this step we return the $s$-th yard $T_s$ and Farey simplex $S_s$, the convex hull of vertices $V_1,\ldots,V_n$ and $V_j',V_1,\ldots,V_n$ respectively,
		\[
		\begin{aligned}
			T_s&= \conv(V_1,\ldots,V_n),\\
			S_s&= \conv(V_j',V_1,\ldots,V_n).
		\end{aligned}
		\]
	\end{mybox}

	\vspace{1mm}
	
	If the Farey summation continued fraction is finite then the algorithm is iterated until the last element $a_j$, at which point all vectors except $V_j$ are set to be the zero vector.
	The algorithm is iterated indefinitely otherwise. 
	It generates the sequence of Farey simplices $(S_s)$ that are used in the Farey summation algorithm for the vector $v$ whose continued fraction is $\alpha$.
	
	As in the Farey summation algorithm, the convex hull of the integer basis vectors $\conv(E_1,\ldots,E_n)$ is the \textit{deck}.
	If the continued fraction is finite then the final yard is the \textit{crow's nest}, and the final vertex generated by the algorithm is the \textit{pennant}

	\begin{example}\label{(5,7,8)-stretching}
		Let us study the case $v=(5,7,8)$ of Example~\ref{(5,7,8)-part 1}, which has continued fraction $[1;1:2|_2\, 1]$.
		We write the corresponding nose stretching algorithm for the basis vectors. 
		For brevity we write the three vectors in the form of a $3\times3$ matrix and omit the Farey simplices.
		
		\[
		\begin{aligned}
			&\left(
			\begin{matrix}
				1 & 0 & 0 \\
				0 & 1 & 0 \\
				0 & 0 & 1
			\end{matrix}
			\right)
			\xrightarrow{1}
			\left(
			\begin{matrix}
				1 & 0 & 0 \\
				1 & 1 & 0 \\
				1 & 0 & 1
			\end{matrix}
			\right)
			\xrightarrow{1}
			\left(
			\begin{matrix}
				1 & 1 & 0 \\
				1 & 2 & 0 \\
				1 & 2 & 1
			\end{matrix}
			\right)
			\xrightarrow{2}
			\\
			&\left(
			\begin{matrix}
				1 & 1 & 4 \\
				1 & 2 & 6 \\
				1 & 2 & 7
			\end{matrix}
			\right)
			\xrightarrow{|_2}
			\left(
			\begin{matrix}
				1 & 0 & 4 \\
				1 & 0 & 6 \\
				1 & 0 & 7
			\end{matrix}
			\right)
			\xrightarrow{1}
			\left(
			\begin{matrix}
				5 & 0 & 4 \\
				7 & 0 & 6 \\
				8 & 0 & 7
			\end{matrix}
			\right)
			\xrightarrow{(|_3)}
			\left(
			\begin{matrix}
				5 & 0 & 0 \\
				7 & 0 & 0 \\
				8 & 0 & 0
			\end{matrix}
			\right).
		\end{aligned}
		\]		
		As output we have the only non zero vector, $v=(5,7,8)$.
		We omit the final $|_3$ in the continued fraction.

	\end{example}

	\begin{definition}
		\label{definition: first convergents}

		Consider a continued fraction $[a_1;\cdots:|_{j_1}:\cdots:a_N]$.		
		On the $i$-th step of the Meester algorithm we get the continued fraction $[a_1;\cdots:a_i]$ and the vector $V_i=(v_{i,1},\ldots,v_{i,n})$.
		The pennant of the Farey polyhedron generated by the nose stretching algorithm applied to $[a_1;\cdots:a_i]$ is called the \textit{$i$-th convergent}.		
		We call $V_i$ the \textit{remainder}.

		\begin{proposition}
			The remainder $V_i$ is the pennant of the polyhedron generated by the nose stretching algorithm applied to
			\[
			[|_{j_1}\cdots |_{j_k}0:\cdots :0 :a_{i+1}:\cdots:a_N].
			\]
			Here the number of initial zeros is the number required to reach the index from Step~$i{+}1$ of the algorithm.
			\qed
			
		\end{proposition}	
		
	\end{definition}
	
	\begin{example}
		For the point $(55,10,67)$ with continued fraction $[0;\, 5: 0: 2|_2\, : 0:1:2:2]$, the remainder when $i=6$ is $V_6=(5,0,2)=[|_2\, 0:2:2]$.
		Note that the index on the~$6$-th step of the algorithm is~$1$.
		
	\end{example}

	\subsection{A few words on convergence of the Farey summation algorithm}
	\label{Convergence of the Farey summation algorithm}
	
	Let us briefly discuss the convergence properties of the Farey summation algorithm. Without loss of generality we restrict to the three-dimensional case, the situation in the higher-dimensional case is similar to~$\r^3$. 
	The results of this section follow directly from~\cite{KM1995}.
	
	\subsubsection{The divergence set is everywhere dense}
	\label{subsection: divergence}

	First of all note that the Farey tessellation
	converges to a single ray if and only if the Meester algorithm converges to $(0,0,0)$.
	It is interesting to observe that the Meester algorithm does not always converge to $(0,0,0)$ everywhere in the positive orthant (see~\cite{MN1989} for the three-dimensional case; see~\cite{KM1995} for the proof in higher-dimensions).
	
	\vspace{2mm}
	
	As it was shown in~\cite{KM1995}
	if at some iteration of the algorithm one of the coordinates exceeds the sum of the other two, then there is no convergency.
	For this reason 
	the convergency set can be constructed by removing the 
	``corner cones'' generated by the vectors:
	\[
	[a_1;\cdots:a_n\,|\,1], \quad 
	[a_1;\cdots:a_n:0\,|\,1], \quad 
	[a_1;\cdots:a_n:0:0\,|\,1];
	\]
	and the ``centre'' is at 
	\[
	[a_1;\cdots:a_n\,|\,].
	\]
	for all admissible sequences $(a_1,\ldots,a_n)$ with $a_n\ne 0$.
	
	\vspace{2mm}
	
	\begin{figure}[t]
		\[
		\includegraphics[width=6.5cm]{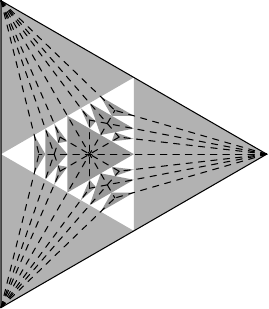}
		\qquad
		\qquad
		\includegraphics[width=6.5cm]{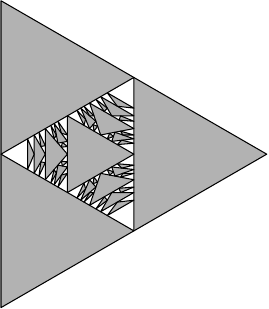}
		\]
		\caption{Meester algorithm divergence set (several iterations)}
		\label{ser-far-1-2.pdf}
	\end{figure}
	
	It is clear that the set is invariant under the multiplication of all the coordinates by any positive real number. Hence to get the structure of the non-convergency set it is sufficient to know the intersection of this set with the plane $x_1+x_2+x_3=0$.
	
	\vspace{2mm}
	
	In Figure~\ref{ser-far-1-2.pdf}
	we show the first several removed triangles.
	They correspond to the sum of the elements of the Farey summation continued fractions smaller or equal to 3 and 4 respectively.
	On the left figure we show the removed (corner) triangles.
	Dashed lines connect the vertices 
	$(1,0,0)$, $(0,1,0)$, and $(0,0,1)$ to the ``centre'' points 
	$[a_1;\cdots:a_n\,|\,]$.

	\subsubsection{No guaranteed algebraic cubic periodicity of the Farey summation algorithm}

	First of all let us observe the following straightforward statement.
	
	\begin{proposition}
		Let $v=(x,y,z)$ be a vector such that $0<x<y<z$.
		Assume that the fraction for $v$ contains only entries 
		corresponding to the first and second coordinates.
		Namely, it is equivalent to the following infinite sequence. 
		\[
		[a_1;a_2:0:a_3:a_4:0\cdots|]
		\]
		Then the regular continued fraction for $y/x$ is
		$
		[a_1;a_2:a_3:a_4\cdots].
		$

		\qed
	\end{proposition}

	Now let us consider one particular example when cubic periodicity fails (that provides a counterexample to Jacobi's last theorem for the Farey summation algorithm).

	\begin{example}{\bf (Failure of algebraic cubic periodicity.)} 
		Consider the matrix 
		\[
		\left(
		\begin{matrix}
			10& 0& 1 \\
			1&  10& 0 \\
			0& 1& 0 \\    
		\end{matrix}
		\right).
		\]
		Its characteristic polynomial
		$t^3 - 20t^2 + 100t - 1$ 
		is irreducible over $\mathbb Q$.
		Let us take the vector $v$ corresponding to the maximal eigenvalues and with the last coordinate equal to $1$:
		\[
		v\approx (3.21113935\ldots,10.31141595\ldots,1).
		\]
		The corresponding Farey summation continued fraction is
		\[
		[0;0:3:4:0:1:2:0:1:3:
		0:1:3:0:1:1:0:2:1:0:1:19:\cdots\,|\,],
		\]
		which is not periodic, since the sequence of elements for this continued fraction coincide (after removing zeros) to the  regular continued fraction for the cubic irrational $v_1/v_3$:
		\[
		\frac{v_1}{v_3}=[3; 4: 1: 2: 1: 3: 1: 3: 1: 1: 2: 1: 1: 19:\cdots].
		\]
		This continued fraction is not periodic by Lagrange's theorem, since the ratio $v_1/v_3$ is not a quadratic irrational number.
	\end{example}

	\section{Properties and invariants of Farey polyhedra}
	\label{section: properties of farey polyhedra}
	Now we come to the main section of the paper. 
	
	\vspace{1mm}
	We start with a short discussion of the basic properties of the Farey tessellation.
	Then in Subsection~\ref{subsection: prismatic triangulation} we introduce a central object of study, the combinatorial \textit{prismatic diagram}.
	From these diagrams, in Subsection~\ref{Sails} we define a generalisation of the notion of sails and introduce the important invariants, the \textit{LLS sequences}.
	
	\vspace{1mm}
	In Subsection~\ref{subsection: Semi-group of matrices} we study the matrix decomposition of the Farey summation algorithm, which is necessary for the study of LLS sequences in Subsection~\ref{LLS sequence and the elements}.
	The decomposition allows a simple definition of three-dimensional \textit{continuants}, shown in Subsection~\ref{subsection: continuants}.
	
	\vspace{1mm}
	We finish the section in Subsection~\ref{LLS sequence and the elements} where we introduce \textit{frieze patterns} from prismatic diagrams, and discuss the generating \textit{Ptolemy relation} for frieze patterns.

	\subsection{Basic properties of Farey tessellation}
	\label{subsection: basic properties tessellation}

	Now it is time to formulate and prove the following general statements of the structure of the Farey tessellation of the positive orthant.
	
	\begin{proposition}\label{basic-properties-FT-extended}
		The following holds:
		\begin{itemize}
			\item[\emph{(}i\emph{)}] The union of all vertices of all Farey simplices $($of all dimensions$)$  is the intersection of the unit integer sphere with the positive orthant;
			
			\item[\emph{(}ii\emph{)}] The closure of the union of all the simplices is the positive orthant;
			
			\item[\emph{(}iii\emph{)}] The interiors of two distinct Farey simplices of maximal-dimension do not intersect.
		\end{itemize}
	\end{proposition}
	
	\begin{proof}
		(\emph{i})
		Consider $v\in \z S^{n-1}$ in the positive orthant. The  Meester algorithm for $v$ will generate the Farey summation continued fraction.
		The corresponding nose stretching procedure for that continued fraction will produce a sequence of Farey simplices, such that the last one contains $v$ as a vertex.
		
		From Corollary~{\ref{cor: farey simplex = basis}} the vertices of any side of any Farey simplex form a basis of the integer lattice.
		Hence it is clear that any vertex of a Farey simplex has integer distance $1$ with the origin, and so is in $\Z S^{n-1}$.
		
		\vspace{2mm}
		
		\noindent
		(\emph{ii})
		Consider any point $v$ with non-rational coordinates and consider any sequence of elements $v_i\in \z S^{n-1}$ whose directions converge to the direction of $v$. 
		Then the closure of the union of segments $Ov_i$ contains $v$ in the closure. 
		Finally note that $Ov_i$ is covered by Farey simplices, where this covering is produced by the nose stretching procedure for $v$.

		\vspace{2mm}

		Since the set of $v$ with irrational coordinates is everywhere dense, the closure of Farey simplices covers the whole positive orthant.
		
		\vspace{2mm}
		\noindent
		(\emph{iii})
		Let the distinct Farey simplices $S_1$ and $S_2$ be constructed by two distinct sequences of the Farey algorithm $\Sigma_1$ and $\Sigma_2$. Let us consider the following cases:
		
		\begin{itemize}
			\item 
			{\it \underline{If $\Sigma_1$ starts from $\Sigma_2$.}}
			Then let us pick the plane $\pi$ of the last exterior face in $\Sigma_1$ through which the ray in the direction of $v$ passes. 
			By construction the origin and $S_1$ are in one halfspace with respect to $\pi$ while $S_2$ is in the other. 
			Therefore, their interiors do not intersect.
			
			\vspace{1mm}
			
			\item
			{\it \underline{If $\Sigma_2$ starts from $\Sigma_1$.}} This case is similar to the above one.
			
			\vspace{1mm}
			
			\item {\it \underline{Neither of the above two cases.}} This means that after some time the sequences of simplices projects to different non-intersecting triangles in the Farey net. Therefore they do not have a common point in their interiors.
			
		\end{itemize}
	\end{proof}

	We would like to continue with the following important example.
	
	\begin{example}\label{example-not all tetrahedra}
		\textbf{Farey tessellation in three-dimensions does not contain all empty tetrahedra of volume 2.}
		Let 
		\[
		v_1=(6,14,15), \quad 
		v_2=(5,13,14), \quad 
		v_3=(5,12,13), \quad 
		\]
		and let $w=v_1\oplus v_2\oplus v_3=(16,39,42)$.
		
		The tetrahedron $wv_1v_2v_3$ is empty.
		The continued fraction of $w$ is $[2;1:2:0:3|_{1,3}]$.
		The base triangle of the tetrahedron that has pennant $w$ is formed from the convergents of $[2;1]$, $[2;1:2]$, and $[2;1:2:0:2]$.
		These are not the vertices $v_1$, $v_2$, and $v_3$.
		The reason for this is that to reach any $v_i$ we must use a two-dimensional step.
		Hence not all empty tetrahedra are contained in the three-dimensional Farey tessellation.
	\end{example}

	\subsection{Prismatic triangulations}
	\label{subsection: prismatic triangulation}
	Now we study the combinatorial structure of Farey polyhedra.

	\subsubsection{Ordered path-triangulations and their complete invariant}
	
	In this section we consider {\it polyhedra} as convex hulls of finite numbers of points. 
	We say that a polyhedron is {\it marked} if each of its edges contain a finite number of marked interior points (zero is also allowed). 
	A decomposition of a marked polyhedron $P$ into non-intersecting simplices (of maximal-dimension) is said to be a {\it triangulation} of $P$
	if the set of vertices of all such simplices coincides with the union of all vertices of $P$ and all marked points of $P$.
	
	\vspace{2mm}
	
	Recall that the {\it dual} graph of a triangulation is the graph whose vertices are labeled by simplices in the triangulation; an edge of the dual graph connects two vertices if and only if the corresponding two simplices share a face of codimension~$1$.  
	
	\begin{definition}
		\label{definition: similar triangulations}
		We say that two triangulations are {\it similar} if there is a one-to-one map between their simplices that provides combinatorial equivalence of their dual graphs. 
	\end{definition}
	
	Below we study triangulations of the following type.
	
	\begin{definition}
		We call a triangulation $T$ a {\it path-triangulation} if its dual graph is a path graph.
	\end{definition}
	
	\subsubsection{Decks, masts, yards, crow's nests, and pennants}

	Recall the definitions of Subsection~\ref{subsection: farey summation algorithm}.
	In order to remove some natural symmetries of path-triangulations we introduce the following definition.

	\begin{definition}
		\label{def: path triangulation}
		Consider a path-triangulation $T$.
		Let us mark by $S_0$ and $S_1$  the two simplices corresponding the endpoints of the dual graph, and let $F_0$ and $F_1$ be one of the codimension~$1$ {\it exterior} faces of $S_0$ and $S_1$ respectively, (i.e.  $F_0$ and $F_1$ are not faces of some simplex other than $S_0$ and $S_1$). We say that $T$ is 
		{\it ordered} if the faces $(F_0,F_1)$ are fixed and the vertices of $F_0$
		are ordered. 
		We say that  $F_0$ is {\it deck} of $T$, and that $F_1$ is the {\it nest} of $T$ (or \textit{crow's nest}).
		The unique vertex of the nest that is not a vertex of any other simplex is called the \textit{pennant}.
	\end{definition}
	
	Here and below we assume all triangulations are ordered unless stated otherwise.
	As we show later the nest has a natural ordering induced by the ordering of its deck
	(see Remark~\ref{nest-ordering-remark} below).

	\begin{definition}
		We say that an edge
		of an ordered path-triangulation is a {\it mast edge} if it is adjacent to a single simplex of maximal-dimension and does not belong either to the deck or the nest of the triangulation. 
		\vspace{1mm}
		
		\noindent
		An edge is said to be a {\it yard edge}
		if it is adjacent to more than one simplex of the triangulation.
		
		\vspace{1mm}
		
		\noindent
		A connected component of the union of edge masts is called a {\it mast}.
		A ($k{-}1$)-dimensional face belonging to several simplices is called a {\it yard}.
	\end{definition}
	
	\begin{example}
		Figure~\ref{figure: prismatic diagram 2d} shows a combinatoric representation of two-dimensional Farey summation continued fractions.
		On these diagrams the deck and nest are represented by grey edges and the masts by the exterior edges.
		The interior diagonals are yard edges (in two-dimensions yards and yard edges coincide).
		The pennant is the point $(3,4)$ in each diagram.
		
		\begin{figure}
			\centering
			\begin{minipage}[ht]{.32\textwidth}
				\centering
				\includegraphics[width=0.65\textwidth]{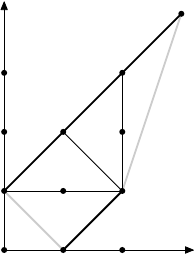}
			\end{minipage}
			\begin{minipage}[ht]{.32\textwidth}
				\centering
				\includegraphics[width=0.65\textwidth]{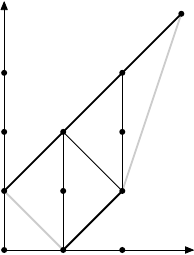}
			\end{minipage}
			\begin{minipage}[ht]{.32\textwidth}
				\centering
				\includegraphics[width=0.65\textwidth]{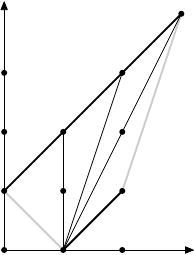}
			\end{minipage}
			\caption{Path triangulations}
			\label{figure: prismatic diagram 2d}
		\end{figure}
	\end{example}
	
	\begin{proposition}\label{masts-decks-prop}
		The following statements hold:
		\begin{itemize}
			\item Each simplex that is not the deck or the nest has precisely one mast edge;
			
			\item Any mast is a broken line;
			
			\item One endpoint of any mast is at the deck and another is at the nest;
			
			\item Any vertex of the nest/deck has at most one mast adjacent to it with one exception. If a vertex of a deck is a vertex of a nest, then no mast is adjacent to this vertex.
		\end{itemize}
	\end{proposition}
	
	\begin{proof}
		Each simplex that is not the deck or the nest of the triangulation contains two faces of codimension~1 that are yards. The union of their edges are all edges of the simplex but one. This concludes the proof of the first item.

		\vspace{2mm}
		
		We prove the second, third, and fourth items simultaneously by induction on the number of simplices in the triangulation.
		
		\vspace{1mm}
		
		\noindent
		{\it Base of induction.} All the statements hold for a single simplex triangulation.
		It has a deck, a nest, and a single mast edge between the vertex of a deck that is not in the nest and a vertex of the nest that is not in the deck.
		
		\vspace{1mm}
		
		\noindent
		{\it Step of induction.} Let the statement hold for all triangulations on $n{-}1$ simplices. Consider any triangulation $T$ of $n$ simplices. The last simplex $S_n$ of this triangulation has the nest $N$ and one yard $Y$. Let $D$ denotes the deck of $T$. 
		
		Let us remove $S_n$ and consider the last yard $Y$ as the nest of a smaller triangulation $T'$. All the statements hold for $T'$.
		Let us now add $S_n$. The last mast edge connects $N$ and $Y$.
		Therefor:
		
		\begin{itemize}
			\item 
			All mast are broken lines, as we add the last edge to the vertex where either only one mast or no masts.
			
			\item We either have changed only one endpoint of masts, which will have vertex in the nest $N$, or we create a new mast, then the corresponding vertex of $Y$ should have been a vertex of the deck $D$ by the induction assumption.
			
			\item Finally $T$ and $T'$ has the same deck $D$ and the nests $Y$ and $N$ different by one vertex connected by a new edge. Hence the last statement holds.
		\end{itemize}
		
		That concludes the proof of the last three items.
	\end{proof}
	
	Proposition~\ref{masts-decks-prop}
	allows us to give  the following definition.
	
	\begin{definition}
		The set of masts admits an {\it induced ordering} by the indices of the first vertices in the masts (which is in the deck $D$ with ordered vertices).
	\end{definition}
	
	\begin{remark}\label{nest-ordering-remark}
		In particular the induced ordering of masts itself induces an ordering on the nest.
		Informally speaking, masts have a flavour of the notion of the parallel transform in differential geometry.
	\end{remark}

	\subsubsection{Prismatic polyhedra and diagrams}
	We introduce prismatic diagrams that encode the combinatorial structure of Farey polyhedra.
	
	\begin{definition}
		Let $\mathcal R^{k-1}$ be a simplex of dimension $k{-}1$ with enumerated vertices $R_1,\ldots,R_k$.
		Let $v$ be a non-zero vector orthogonal to the plane $R_1,\ldots, R_k$.
		Let also $D=(d_1,\ldots, d_k)$
		be a sequence of non-negative integers.
		A {\it prismatic polyhedron} is a marked polyhedron
		\[
		\conv(R_1,\ldots, R_k, R_1+d_1v,\ldots, R_k+d_kv),
		\]
		where at each edge $(R_i, R_i{+}d_iv)$ all the points $R_i{+}kv$ with $1\le k\le d_i{-}1$ are marked.
		\\
		We denote it by $(\mathcal R^{k-1},v,D)$.
		
		\vspace{2mm}
		
		\noindent
		For any triangulation of the prismatic polyhedron we mark the faces 
		\[
		F_0=(R_1,\ldots, R_k),
		\quad \hbox{and} \quad 
		F_1=(R_1+d_1v,\ldots, R_k+d_kv);
		\] 
		we order the vertices of faces as indices of $R_i$. 
		Any ordered triangulation (with marked faces $F_0$ and $F_1$ as above) of a prismatic polyhedron is called a {\it prismatic diagram}.
		A prismatic diagram with $n$ simplices (of maximal-dimension) in its triangulation is said to have \textit{length $n$}.
	\end{definition}
	
	In the two-dimensional case the canonical prismatic diagram is also known under the name \textit{Farey boat}, based on the terminology from~\cite{moriergenoud2019}.

	\begin{definition}
		A prismatic diagram is said to be {\it canonical} if $\mathcal R^{k-1}$ is the simplex whose vertices are the endpoints of the first $k$ coordinate vectors and $v=(1,1,\ldots, 1, 0,0,\ldots, 0)$, where the number of unit coordinates is $k$. 
	\end{definition}

	\begin{definition}
		\label{definition: LR sequence}
		Consider a canonical prismatic triangulation
		$\mathcal D$ of length $n$ and dimension $k$. The {\it LR-sequence} of $\mathcal D$ is the following sequence of indices in $\{1,\ldots,k\}$ of length $n$:
		\[
		(M_1,\ldots, M_n),
		\]
		where $M_i\in\{1,\ldots,k\}$ denotes the index of a mast edge that we build on Step~$i$.
	\end{definition}
	\begin{remark}
		The term \textit{LR-sequence} comes from the theory of regular continued fractions.
		\textit{LR} stands for \textit{Left-Right}: one can view the coefficients of a regular continued fraction of $\alpha$ as the number of consecutive left or right turns in the cutting sequence of $\alpha$.
		See~\cite{Series1985} for a nice visualisation.
	\end{remark}
	Recall that the masts are ordered $(1,\ldots,k)$: at each Step $i$ of the Farey summation algorithm we add a single mast edge to one of the masts. 
	The index of the addended mast is collected as $M_i$.

	\begin{remark}
		\label{remark: LR exponent}
		In particular we can write the LR-sequence in the exponential form: 
		\[
		1^{a_1}2^{a_2}\ldots k^{a_k}1^{a_{k+1}}\ldots
		\]
		for some non-negative collection of $a_i$.
		This expression provides a link to theory of ordinary continued fractions, which we explore in this paper for the particular case of Farey summation continued fractions. (It can be applied to other additive algorithms as well.) 
	\end{remark}
	
	\begin{example}
		In Figure~\ref{figure: prismatic diagram 2d} we show the three triangulations of the prismatic polyhedra in two-dimensions with $D=(1,3)$.
		The sequence $D$ encodes the number of mast edges in each mast.
		The LR-sequences from left to right are $(1,2,2,2)$, $(1,2,1,1)$, and $(1,1,1,2)$ or, in exponential form, $(1^1,2^3)$, $(1^1,2^1,1^2)$, and $(1^3,2^1)$, respectively.

	\end{example}
	
	\begin{remark}
		\label{remark: combinatoric picture}
		
		Canonical prismatic diagrams generalise the Farey boats introduced by S.~Morier-Genoud and V.~Ovsienko in~\cite{moriergenoud2019} to higher-dimensions.
		The prismatic diagrams in Figure~\ref{figure: prismatic diagram 2d} are exactly the Farey boats (also called \textit{wrinkled triangulations}), after transformation of vertices by 
		\[
		\left(
		\begin{matrix}
			0 & 1 \\
			-1 & 1
		\end{matrix}
		\right).
		\]
		The prismatic diagrams give a combinatorial description of Farey polyhedra.
		In two-dimensions they link to the theory of cluster algebras through their relation to Conway-Coxeter friezes.
		The connection between frieze patterns and triangulated polygons was found by J.~Conway and H.~Coxeter in~\cite{ConwayCoxeter1973}, and the connection to cluster algebras by P.~Caldero and F.~Chapoton~\cite{calderochapoton2006}. 
		
	\end{remark}
	
	The following statement is now straightforward.
	
	\begin{corollary}
		The LR-sequence is an invariant of
		similarity of ordered path-triangulations.
		\qed
	\end{corollary}

	Let $T$ be a path-triangulation of a $k$-dimensional polyhedron.
	Let us describe a natural piecewise linear map to the canonical prismatic diagram $T$.
	
	\begin{itemize}
		\item[--] First of all we map the deck of $T$ linearly to the polyhedron $\conv E_1\ldots E_k$.
		\item[--] Secondly we map the masts linearly to the corresponding lines $E_i+tv$ parametrised by $t$. The map is linear at each mast edge, and it sends consequences vertices to consequent points $E_i+jv$ where $j=1,2,\ldots$. 
		\item[--] Finally, once the images of all simplices are defined, we map them linearly as well. 
		
	\end{itemize}

	The next definition extends Notion~E (prismatic diagrams) from the two-dimensional case of Section~\ref{Two-dimensional case}.

	\begin{definition}\label{pathtriangulated}
		The prismatic polyhedron constructed above is said to be {\it associated} to a path-triangulation $T$.
		\\
		The associated polyhedron has an induced structure of an ordered triangulation,
		which we call the {\it canonical prismatic diagram} of $T$ and denote by $\mathcal D(T)$.
	\end{definition}

	\begin{definition}
		We say that two prismatic diagrams are similar if there exists an affine map that preserves  the triangulation and sends marked faces $F_0$ and $F_1$ from the first diagram to the marked faces $F_0'$ and $F_1'$ respectively (here our map must preserve the enumeration of vertices) of the second diagram. 
	\end{definition}
	
	\begin{corollary}
		The canonical prismatic diagram
		is a complete invariant of 
		prismatic triangulations.
	\end{corollary}
	
	\begin{proof}
		Indeed any canonical prismatic diagram is an ordered triangulation, so all canonical prismatic diagrams appear.
		
		It is also clear from construction that different canonical prismatic diagrams are not similar.
		
		Finally the above construction shows that every ordered triangulation is similar to one of the canonical prismatic diagrams.
	\end{proof}
	
	\begin{remark}
		As a conclusion the LR-sequence is a complete invariant of canonical prismatic diagrams (and, equivalently, of similarity types of ordered path-triangulations
		).
		Hence the number of distinct canonical prismatic diagrams (of dimension $k$ consisting of $d$ simplices) coincides with the number of LR-sequences of length $d$, which is equal to $k^d$.
	\end{remark}

	\subsubsection{Prismatic flag diagrams}
	
	We introduce prismatic flag diagrams.
	These flag diagrams allow us to knit together the prismatic diagrams of segments of Farey polyhedra separated by a dimension drop.
	
	\begin{definition}
		Let $P_1(\mathcal R^{k_1-1},v_1,D_1)$
		and $P_2(\mathcal R^{k_2-1},v_2,D_2)$
		be two prismatic polyhedra
		such that 
		$k_2\le k_1$.
		Let also $S$ be a $k_2$ element sequence of indices from $\{1,\ldots,k_1\}$.
		The {\it concatenation of $(\mathcal R^{k_1-1},v_1,D_1)$ and $(\mathcal R^{k_2-1},v_2,D_2)$ with respect to $S$} is adding extra segments in the direction of the rays with indices $s_i\in S$ for $i=1,\ldots, k_2$:
		\[
		[R_{1,s_i}+d_{1,s_i}v_1,R_{1,s_i}+d_{1,s_i}v_1+d_{2,i}v_1],
		\]
		where $s_i$ is the $i$-th element of $S$.
		The resulting set is the union of the polyhedron $P_1$ and the convex hull of all the new added segments 
		\[
		P_1+_SP_2.
		\]
		We denote it by
		\[
		\big(\mathcal R^{k_1-1},v_1,D_1,(D_2,S)\big).
		\]
	\end{definition}
	
	\begin{definition}
		Now let us have a sequence of prismatic polyhedra $P_j(\mathcal R^{k_j-1},v_j,D_j)$
		with $j=1,\ldots, l$ with  $k_1>k_2>\cdots>k_l$.
		Let also $S_i$ be sequences of indices of $k_i$ elements
		such that, considered as sets, 
		$S_l\subset S_{l-1}\subset \cdots \subset S_1$. Here we assume $S_1=(1,2,\ldots, k_l)$.
		The polyhedron
		\[
		(\ldots (P_1+_{S_2} P_2) +_{S_3} P_3) +_{S_4} \cdots 
		)+_{S_l}P_l.
		\]
		is said to be a {\it prismatic flag polyhedron}.
		We denote it by
		\[
		\big(\mathcal R^{k_1-1},v_1,(D_i,S_i)_{i=1}^{l}\big).
		\]
		Let now  $P_j$ have prismatic diagrams $T_j$ for $j=1,\ldots, l$. 
		The natural mapping of these triangulations to the prismatic flag polyhedron is said to be a {\it prismatic flag diagram}. 
		We denote it by
		\[
		(T_i,S_i)_{i=1}^l,
		\]
		where $S_1=(1,\ldots,k)$.
		
		\noindent          
		In the case when $T_1$ is a canonical prismatic diagram, we say that the prismatic flag diagram is {\it canonical}  as well. 
	\end{definition}
	
	\begin{remark}
		Note that the concatenation of prismatic polyhedra and their prismatic diagrams depends neither on the positions of simplices $R^{k_i-1}$ nor on vectors $v_i$ for $i=2,\ldots, l$; it depends only on the dimensions $k_i$ and the sets $D_i$
		(for $i=2,\ldots, l$) and on 
		the first prismatic polyhedron $(R^{k_1-1},v_1,D_1)$.
	\end{remark}
	
	\begin{remark}
		Finally it remains to say that, due to the freedom of the ordering of elements within the sequences $S_i$, 
		there are several obvious ways to obtain the same canonical flag diagrams (by permuting $S_i$ and the LR sequences of the corresponding $T_i$). 
	\end{remark}
	
	\subsubsection{Canonical Prismatic flag diagrams and Farey summation continued fractions}
	
	Consider an extended Meester continued fraction:
	\[
	\alpha=\big[a_1;\cdots: a_{s_{j_1}} \,|_{j_1}\,
	a_{s_{j_1}+1}: \cdots :
	a_{s_{j_2}} \,|_{j_2}\,
	a_{s_{j_2}+1}:\cdots : 
	a_{s_{j_{n-2}}} \,|_{j_{n-2}
	}\,
	a_{s_{j_{n-2}+1}}: \cdots : a_N\big].
	\]
	
	For every $i=1,\ldots,n-1$ we let 
	$T_i$ be the canonical prismatic diagram
	whose LR-sequence is 
	$(a_{s_{j_{i-1}}+1},\ldots,a_{s_{j_{i}}})$.
	Here we allow the sequences to be empty. 
	Finally let $S_1=(1,\ldots,k)$ and let $S_i$ be the sequences with ordering induced by $S_1$, whose elements, when considering the $S_i$ as sets, are defined by
	\[
	S_i=S_{i-1}\setminus \{j_i\},
	\]
	for $i=2,\ldots,n-1$.
	
	\begin{definition}
		We say that the canonical flag diagram $(T_i,S_i)_{i=1}^{n-1}$ constructed above is
		the {\it canonical prismatic flag diagram} for 
		the Farey summation continued fraction $\alpha$. We denote it by $T(\alpha)$.
	\end{definition}

	\begin{theorem}
		\label{theorem: prismatic diagrams invariant}
		The canonical prismatic flag diagrams $T(\alpha)$ form a complete invariant for Farey summation
		continued fractions $\alpha$.
		In addition the combinatorics of $T(\alpha)$ coincide with the combinatorics of the
		Farey polyhedron.
		\qed
	\end{theorem}
	
	\begin{proof}
		Invariance follows as every Farey polyhedron admits a canonical prismatic flag diagram via its extended Meester continued fraction.
	\end{proof}
	
	\begin{example}\label{(5,7,8)-prismatic}
		Let us show the prismatic triangulations for the vector $v=(5,7,8)$ of Example~\ref{(5,7,8)-part 1}.
		
		\vspace{2mm}
		
		Note that we have two different continued fractions  defined by $(5,7,8)$:
		\[
		[1;1:2\, |_2\, 1],
		\quad 
		\hbox{and}
		\quad 
		[1;1:2\, |_2\,0:1].
		\]
		(Technically the algorithm will never arrive to the second continued fraction.)
		
		The corresponding canonical prismatic 
		flag diagrams are shown in 
		Figure~\ref{sail3d-3-4.pdf}.
		They both
		consist of $4$ simplices of dimension $3$ and one simplex of dimension $2$.
		The simplices of dimension $3$ are
		\[
		A_0B_0C_0A_1, \quad A_1B_0C_0B_1,
		\quad A_1B_1C_0c,
		\quad A_1B_1c\, C_1.
		\]
		The simplices (triangles) of dimension 2 are $A_1C_1A_2$ and $A_1C_1C_2$ respectively.
		
		\begin{figure}[t]
			\[
			\includegraphics[width=7cm]{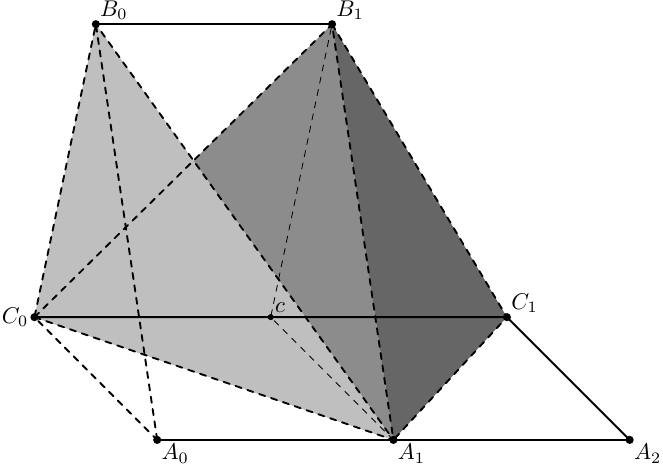}
			\includegraphics[width=7.5cm]{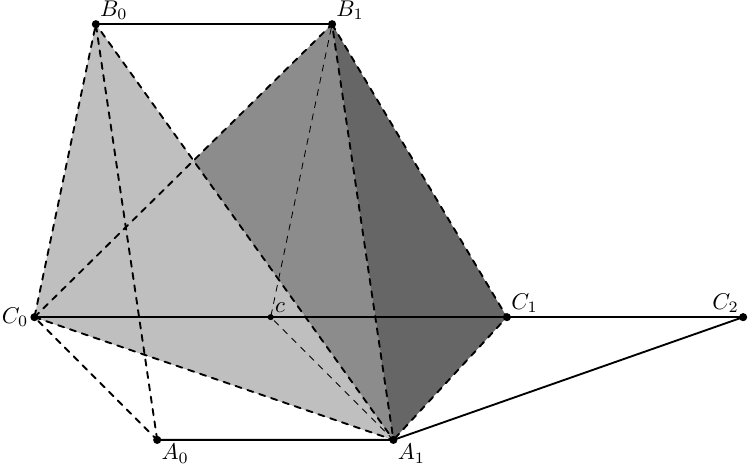}
			\]
			\caption{The two prismatic triangulations for $(5,7,8)$ corresponding to $[1;1:2\, |_2\, 1]$ (Left) and to $[1;1:2\, |_2\,0:1]$ (Right).}
			\label{sail3d-3-4.pdf}
		\end{figure}
	\end{example}
	
	\begin{remark}
		The above construction has a straightforward generalisation to the case of infinite continued fractions. We omit it here.
	\end{remark}
	
	\subsection{Sails and their LLS sequences}
	\label{Sails}
	
	In this subsection we discuss a generalisation of sails and LLS sequences to the multidimensional case.

	\subsubsection{Sails of prismatic diagrams and Farey polyhedra}

    The next definition extends Notion~C from two-dimensional case of Section~\ref{Two-dimensional case}.	
	
	\begin{definition}\label{3d_saiL}
		Let $i\in\{1,\ldots,n\}$ and let $T$ be a canonical prismatic flag diagram in dimension $n$. 
		Consider the hyperplane spanned by the vector $(1,1,\ldots,1)$ and a codimension $1$ face of the deck (where we exclude the $i$-th vertex of the deck).
		The intersection of $T$ with this hyperplane is called the {\it sail} of the $T$ and denoted $\sail_i(T)$. 
		
		\vspace{1mm}
		
		\noindent
		The intersection of the sail $\sail_i(T)$ with the deck, the nest, masts, and yards are respectively 
		the {\it deck}, the {\it nest}, 
		the {\it masts}, and the {\it yards}
		of $\sail_i(T)$.
	\end{definition}
	
	The sails of the canonical prismatic flag diagram for the vector $(5,7,8)$ are shown below in Example~{\ref{(5,7,8)-sails}}.
	
	\begin{definition}
		The inverse image of the piecewise linear map 
		(see Definition~\ref{pathtriangulated})
		between the union of the
		Farey polyhedron (equipped with its subdivision to Farey simplices) and the corresponding canonical prismatic flag diagram defines the {\it masts} of the Farey polyhedron.
		Recall that the deck/nest/yards/pennant 
		are already stated in Algorithm~\ref{definition-farey-summation-algorithm}.

		\vspace{1mm}
		
		\noindent
		In addition the inverse image defines {\it sails} of the Farey polyhedron and their {\it decks}, {\it nests}, 
		{\it masts}, and {\it yards}.
	\end{definition}

	\begin{remark}
		\label{remark: classical sail}
		The sail $\sail_i(T)$ is the union of the simplices bound by all but the $i$-th mast, and the masts themselves.
		In two-dimensions $\sail_i(T)$ is simply the $i$-th mast.
		In this case we have a classical construction of sail, see, for example, in~\cite{karpenkov-book-v2}.
		The earliest use of the term sail is from V.~I.~Arnold, see for instance~\cite{Arnold2002}.
		The theory of geometric continued fractions was instigated by F.~Klein in his papers~\cite{Klein1895,Klein1896}.
		
		Three-dimensional {\it sails} for a ray $v$ in the positive octant are the unions of pairs of masts together with the sequence of triangles joining them.
	\end{remark}

	\subsubsection{On the masts at the point of dimension drop}
	
	First let us assume that all of the simplices of the Farey summation continued fraction are
	\label{Farey polyhedron, or cf?? think of def}
	of the same dimension $k$.
	Then all the mast edges are uniquely defined except for the last one.
	This happens since there is no natural way to define the nest combinatorially (without involving any ordering).
	So we can connect the last vector with any of the masts. In other words we have:
	\[
	[a_1; a_2:\cdots :a_n]
	=[a_1; a_2:\cdots :a_n{-}1:1]
	=[a_1; a_2:\cdots :a_n{-}1:0:\ldots: 0:1].
	\]
	where the number of zeros does not exceed $k-2$. 
	This will represent precisely all possible choices of the nest and of the last element of the mast. 
	
	\vspace{2mm}
	
	This corresponds to the phenomenon for classical continued fractions, where:
	\[
	[a_1; a_2:\cdots :a_n]
	=[a_1; a_2:\cdots :a_n{-}1:1].
	\]
	
	\vspace{2mm}
	
	Now if a canonical prismatic diagram consists of several parts of different dimensions, a similar situation occurs. 
	Here, for each one of the masts that vanishes at this step, its indices should be connected to the first new point of the next part of the flag.

	\subsubsection{A rigid structure of Farey masts, nose stretching}
	\label{A rigid structure of Farey masts, nose stretching}

	If we elongate any edge of any mast of a Farey polyhedron by
	a unit integer length vector in the direction of the edge we will reach the first point of some edge on one of the other masts.
    This extends Notion~H from the two-dimensional case of Section~\ref{Two-dimensional case}.
	
	\begin{remark}
		Informally speaking the structure of the Farey polyhedron is linearly rigid.
		It gives rise to various dualities of the sails. 
		The fact that stretching a mast edge provides the starting point of a separate mast edge provides a very fast geometric construction of the Farey polyhedron from a continued fraction.
	\end{remark}
	
	Namely, we start with basis elements $E_i$ ($i=1,\ldots, k$) and their sum $P_1$, and with a continued fraction $[a_1;a_2:\ldots]$.
	Then add the vector $E_1P_1$ to $E_1$, $a_1$ times.
	This will generate the first new vector of the Farey polyhedron (of course if $a_1\ne 0$), which we denote by $E_{1,2}$. 
	
	Now we add $E_1P_1$ to $E_{1,2}$ once more to get $P_2$, symbolically
	\[
	\begin{aligned}
		E_{1,2}&=E_1+a_1 E_1P_1,\\
		P_2&=E_1+(a_1{+}1)E_1P_1.
	\end{aligned}
	\]
	Now the point $P_2$ is on the first edge emanating from $E_2$ (except for the case $a_2=0$).
	So we set
	\[
	\begin{aligned}
		E_{2,2}&=E_2+a_2E_1P_2,\\
		P_2&=E_2+(a_2{+}1)E_1P_2.
	\end{aligned}
	\]
	Proceeding further we will construct the whole Farey polyhedron.
	This procedure is a reinterpretation of nose stretching for Farey polyhedron. 
	Informally, at each moment we stretch one of our $k$-noses further and further away.

	\begin{definition}
		The vertices $E_{i,j}$ generated by the algorithm are called the {\it partial quotients} of $v$.
	\end{definition}

	\begin{remark}\label{Stretching-olny-farey}
		Nose stretching exactly generalises the procedure for classical continued fraction theory.
		This procedure is specific to the Farey algorithm. 
	\end{remark}

	\subsubsection{LLS sequences and their dualities}

	Every sail possesses a collection of invariants that encodes most of the elements of the continued fractions.
	In this subsection we introduce the LLS sequence.
	Let us start with a few notions used in the definition of the LLS sequence.
	
	\vspace{2mm}
	
	An angle is {\it integer} if its vertex is an integer point.
	An angle is rational if both of its edges contain integer points distinct from the vertex.
	
	We say a space is \textit{integer} if it contains a full rank integer lattice.
	\begin{definition}
		
		\noindent Let $\pi_1$ and $\pi_2$ be two integer spaces whose intersection $P_1\cap P_2$ is also an integer space.
		Let $P_1$ and $P_2$ be integer bases for $\pi_1$ and $\pi_2$ respectively such that $P_1\cap P_2$ is a basis for $\pi_1\cap\pi_2$.
		We call this intersection $W=P_1\cap P_2$.
		Let $U=P_1\setminus W$ and $V=P_2\setminus W$.		
		Then the expression
		\[
		\frac{\iv(V,U,W)}
		{\iv(V)\cdot \iv(U)\cdot \iv(W)}
		\]
		is called the \textit{integer sine} of the angle between  $\pi_1$ and $\pi_2$ and denoted by $\isin(\pi_1,\pi_2)$.
	\end{definition}
	
	Recall Definition~{\ref{definition: principal}} of a principal yard and that a division simplex is the convex hull of two consecutive principal yards.
	
	\begin{definition}
		
		Consider the sequence of division simplices $T_i$ constructed by the Farey summation algorithm for some integer vector.
		Let $F_i$ be the intersection of a sail with the division simplex $T_i$. 
		Then we say that $F_i$ is a \textit{principal face of this sail} if
		$\dim F_i\geq \dim T_i-1$. 
		The set of all principal faces of a sail has a natural ordering induced by the ordering of the Farey summation algorithm.
		
		\vspace{1mm}
		
		\noindent
		A {\it mast segment} is the union of all mast edges in a line in a Farey polyhedron.
	\end{definition}
	
	\begin{remark}
		For the principal face $F_i$ there are the following two possibilities.
		\[
		\dim F_i=
		\left\{
		\begin{array}{ll}
			\dim T_i,& \hbox{if a mast was removed on the $(i-1)$-th step;}\\
			\dim T_i-1,& \hbox{otherwise.}\\
		\end{array}
		\right.
		\]
	\end{remark}
	
	\begin{example}
		In Figure~\ref{sail3d-3-4.pdf}, for the sail between masts $A$ and $C$, we see the principal faces are $A_0A_1C_0$, $A_1C_0C_1$, and $C_1A_1A_2$.
	\end{example}

	In the next definition we extend Notion~I from the two-dimensional case of Section~\ref{Two-dimensional case}.
    
	\begin{definition}\label{def:LLS-3D}
		
		The LLS sequence of the $j$-th sail is the $j$-th sail of the canonical prismatic flag diagram 
		equipped with the following:
		
		\begin{itemize}
			
			\item Each pair of consecutive principal faces is equipped with the integer sine between the planes that these faces span.
			We indicate the yard edge connecting both principal faces (edge of a principal yard) and equip it with the corresponding integer sine (a dashed line indicates $0$ integer sine: in this case the principal faces lie in the same plane in the Farey polyhedron);
			
			\item We indicate the dropping of the $j$-th mast by a double yard edge.
			
			\item For the simplicity of drawing we replace mast segments of length $n$ by a single segment of length $1$ equipped with $n$ (the integer length).
			
		\end{itemize}
	\end{definition}
	
	\begin{remark}
		The term LLS sequence (lattice length-sine sequence) comes from a similar notion in the two-dimensional case
		(see, for example, in~\cite{karpenkov2008,karpenkov-book-v2}).
		In the two-dimensional case we have a sequence of numbers, whereas in three-dimensions the prismatic flag diagrams encode sequences of principal faces.
		
	\end{remark}
	
	For consistency, let us start with our standard example.
	
	\begin{example}\label{(5,7,8)-sails}
		Let us start with the Farey polyhedron for $v=(5,7,8)$ of Example~\ref{(5,7,8)-part 1}.

		\[
		\includegraphics[height=4cm]{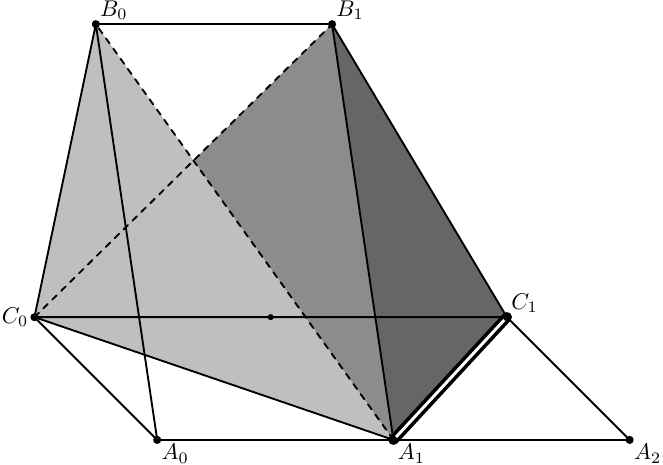}
		\]
		\vspace{2mm}
		
		The sails opposite to masts $B$, $C$, and $A$
		have the following LLS sequences.
		\[
		\includegraphics[height=2.6cm]{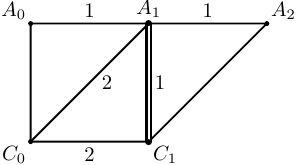}
		\qquad
		\includegraphics[height=2.6cm]{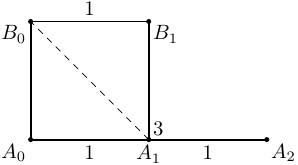}
		\qquad
		\includegraphics[height=2.6cm]{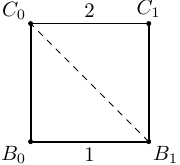}
		\]
		
		The first sail is the most informative,
		it shows all the positive elements of the 
		continued fraction $[1;1:2:0:0 \, |\, 1]$.
		As we will see later in  Proposition~\ref{On duality-prop},
		the first element of the Farey summation continued fraction is the first number in the top row; the second element is a half of the first number in the middle row; the third element  
		is the bottom number; the last element is the last number in the top row.

	\end{example}
	
	In order to further understand the coefficients of the LLS sequence we consider a longer example.
	
	\begin{example}
		Let us consider a vector
		\[
		v=
		(1656812331613081,18353000512178816,19770900109601816).
		\]
		Direct computation shows that its Farey summation continued fraction is
		\[
		[11;12:13:14:15:16:17\,|\, 100: 200: 300: 400].
		\]
		The LLS sequence of the sail for the first two masts is as follows:
		\[
		\includegraphics[height=2.5cm]{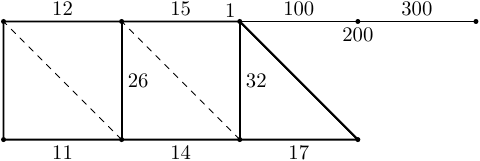}
		\]
		Here the first mast is the bottom one,
		the second mast is the top one.
		We write the corresponding integer lengths near the corresponding mast segments.
		
		\vspace{1mm}
		
		\noindent
		The non-horizontal segments with numbers indicate the common faces of two consecutive principal faces, the numbers are the corresponding integer sines of the angles between the planes.  
		
		\vspace{1mm}
		
		\noindent 
		Finally, dashed lines correspond to the yards joining the endpoints of the mast segments of the same principal face.
		(They indicate the order in which the continued fraction should be taken, here we add triangles from the left to the right). For the numbers on the triangles, we write the mast element first.
		
		\vspace{1mm}
		
		The elements of the Farey summation continued fraction of index $3i+1$ are on the first mast; of index $3i+2$ are on the second. In order to get the remaining elements, we divide the sequence of integer sines by the dimension of the lowest-dimension adjacent principal face, i.e. $26/2$, $32/2$, and $200/1$.  
		
		\vspace{2mm}
		
		Let us write the LLS sequences for the other two sails. The LLS sequence for the sail with the second and the third masts is follows:  
		\[
		\includegraphics[height=2.5cm]{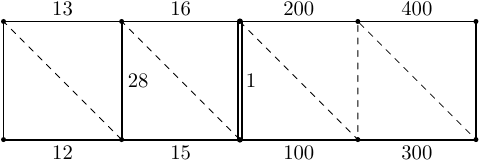}
		\]

		Finally the LLS sequence for the sail with the third and the first masts is
		\[
		\includegraphics[height=2.5cm]{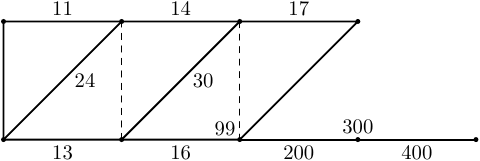}
		\]
		
	\end{example}

	\begin{remark}
		As we see the LLS sequences contain information on most of the elements of the continued fraction (except for the few first and last elements)
		This situation is similar to the classical case. It is due to the duality discussed in the next proposition. 
	\end{remark}
	
	We say that the sail {\it opposite} to a given mast is the sail that does not contain this mast.
	
	We say that two mast segments on different sails of the Farey polyhedron are {\it neighbours} if the lines containing them intersect.
	As we know from Subsection~\ref{A rigid structure of Farey masts, nose stretching}
	every mast segment contains precisely two neighbours (except for the first one and for the last one), and the intersection of the principal faces containing the two neighbours is a yard edge.

	Let us now extend Notion~J from Section~\ref{Two-dimensional case}.
    
	\begin{proposition}\label{On duality-prop}{\bf (On duality of integer lengths and integer sines in sails.)}
		Consider a mast segment $M$ and let $\pi_1$ and $\pi_2$ be the planes of the principal faces of the opposite sail that contain neighbouring segments.
		We assume that the dimensions of $\pi_1$ and $\pi_2$ are both equal to $k$.
		Then
		the integer length of the mast segment $M$ multiplied by $k$ is the integer sine of the angle between $\pi_1$ and $\pi_2$, 
		\[
		\lsin(\pi_1,\pi_2)=k\cdot\il(M).
		\]
	\end{proposition}
	
	\begin{proof}
		The proof is straightforward. 
		Any two neighbouring planes in some coordinates are equivalent (after change of integer basis) to two principal faces of a short Farey summation continued fraction formed only by 3 positive elements, say $(a,b,c)$.
		Here each element corresponds to the integer length of one of the three masts.

		\vspace{1mm}
		
		Note that the elements $a$ and $c$ correspond to some mast segments of the sail, while $b$ corresponds to some mast segment of the opposite mast.
		There are only two different cases here (up to re-numeration of basis vectors):
		either the mast segments for $a$ and $c$ are on the same mast or not (i.e. the continued fraction is $[a;b:c]$ or $[a;b:0:c]$).
		
		\vspace{1mm}
		
		In both cases direct computations show that the integer sine equals $kb$.
		This concludes the proof.
	\end{proof}
	
	We discuss the three-dimensional case in more detail (including the cases of different dimensions of faces) in Subsection~\ref{LLS sequence and the elements} after we introduce the matrix form.

	\subsection{Semi-group of matrices by multiplication}
	\label{subsection: Semi-group of matrices}
	
	The multidimensional Farey summation algorithm can be described by matrix multiplication. 
	Each step of the algorithm defines a vector that is a vertex of the Farey polyhedron.
	The following step proceeds by multiplying this vector by an appropriate matrix.
	In this subsection we continue with the three-dimensional case in order to simplify the exposition.
	Nevertheless most of the definitions, notions, and statements have a straightforward generalisation to the multidimensional case. 
	We also specify the use of Farey summation continued fractions in this subsection (not Meester continued fractions).
	
	\subsubsection{Matrices associated to the steps of the algorithm}

	The algorithm begins in Stage~1 with two-dimensional yards; at that stage we employ  the matrices:
	\[
	A_1=\left(
	\begin{matrix}
		1 & 0  & 0\\
		1 & 1  & 0\\
		1 & 0  & 1 \\
	\end{matrix}
	\right),
	\qquad 
	A_2=\left(
	\begin{matrix}
		1 & 1  & 0\\
		0 & 1  & 0\\
		0 & 1  & 1 \\
	\end{matrix}
	\right),
	\quad \hbox{and} \quad
	A_3=\left(
	\begin{matrix}
		1 & 0  & 1\\
		0 & 1  & 1\\
		0 & 0  & 1 \\
	\end{matrix}
	\right).
	\]
	Once we arrive at one-dimensional yard (i.e. to Stage~2) we continue with matrices
	\[
	B_{ij}=
	\id + E_{i,j}, \quad \hbox{for $1\le i,j\le 3$ and $i\ne j$.}
	\quad 
	\]
	Here $\id$ is the identity matrix and 
	$E_{i,j}$ is a matrix with only one non-zero entry at position $(i,j)$ which is equal to $1$.
	
	\vspace{2mm}
	
	\subsubsection{Partial quotients}
	
	Let us continue with the notion of partial quotients.
	
	\begin{definition}\label{matrix-cf}
		Let $\alpha=[a_1:\cdots:a_n|b_1:\cdots :b_m]$ be a continued fraction. 
		Consider the matrix
		\[
		M=A_1^{a_1}A_2^{a_2}A_3^{a_3}A_1^{a_4}\ldots A_{k}^{a_n} 
		B_{st}^{b_1}B_{ts}^{b_2}B_{st}^{b_3}B_{ts}^{b_4}\ldots,
		\]
		where $k=n \mod 3$, $s=n+1\mod3$, and $t=n+2\mod3$.
		We say that the column vectors of $M$ form the integer basis {\it associated} with the continued fraction.
		We say that the matrix $M$ is the {\it continued fraction matrix} and denote it by $M(\alpha)$.
		We set $M()$ to be the identity matrix.
	\end{definition}
	
	\begin{remark}
		The matrix decomposition of Definition~\ref{matrix-cf} is extremely important. It provides an analytic continuation of the Farey summation algorithm to the case of arbitrary real elements. In the above definition one can take arbitrary real numbers $(a_i)$ and $(b_j)$, as these powers are well-defined for the matrices. Here one should replace the non-zero off-diagonal elements by the corresponding $(a_i)$ and $(b_i)$;   
		e.g.
		\[
		A_1^{a_1}=\left(
		\begin{matrix}
			1 & 0  & 0\\
			a_1 & 1  & 0\\
			a_1 & 0  & 1 \\
		\end{matrix}
		\right).
		\]
	\end{remark}
	
	\begin{remark}
		We would like to mention that the matrix multiplication order for Farey summation algorithm is as in Definition~\ref{matrix-cf} and not the inverse one. 
	\end{remark}

	\begin{definition}
		Let $\alpha$ be some continued fraction and let $\alpha_i$ be the continued fractions
		defined by the first $i$ elements of $\alpha$.
		We say that $\alpha_i$ is the  {\it $i$-th partial quotient} of $\alpha$, and call $M(\alpha_i)$ the {\it $i$-th partial quotient matrix of $\alpha$} and denote it by $M_i(\alpha)$.
	\end{definition}
	
	\begin{proposition}\label{matrix-product}
		Let $\alpha$ be a continued fraction for a vector $v$.
		Then the partial quotient matrix $M_i(\alpha)$ contains all the vectors of the $i$-th yard of the Farey polyhedron of $v$ as columns.
	\end{proposition}
	
	\begin{proof}
		The $M_0()$ matrix contains the standard basis vectors, which form the deck ($0$-th yard).
		Right multiplication of $M_{i-1}$ by $A_j$ replaces the $j$-th column of $M_{i-1}$ by the Farey summation of all vectors in $M_{i-1}$.
		Similarly right multiplication of $M_{i-1}$ by $B_{st}$ replaces the $t$-th vector of $M_{i-1}$ with the Farey summation of the $s$- and $t$-th vectors.
		The proof then follows by induction.
	\end{proof}
	
	In particular if we have the following statement.
	
	\begin{corollary}
		Let $\alpha$ be a finite continued fraction for $v$.
		Then the continued fraction matrix $M(\alpha)$ contains the vector $v$ as a column.
		\qed
	\end{corollary}
	
	\vspace{1mm}
	
	\begin{remark}
		\label{remark: matrix form gives only IV=1 simplices}
		It turns out that matrices $A_1$, $A_2$, and $A_3$ form a freely generated semi-group (as they correspond to different triangles in the tessellation). 
		Their products applied to $(1,0,0)$ generate all integer vectors $v$ of $\z S^2$ whose continued fractions have empty sequence $(b_i)$ (i.e. no dimension drop).
		It is clear that we cannot obtain the integer vectors of $\z S^2$ that have a non-zero $(b_i)$ sequence.
		For instance,  $v=(6,14,15)$ from Example~\ref{example-not all tetrahedra} is one of such vectors, it has continued fraction $[2;1:2|_2\, 0:0:2 ]$.
	\end{remark}
	
	We arrive at the following open question.
	
	\begin{problem}
		Describe the set of matrices $M(\alpha)$ where $\alpha$ has a finite sequence $(a_i)$ and an empty sequence $(b_j)$.  
	\end{problem}
	
	\begin{remark}
		Now we have a simple way to compute all Farey simplices in the Farey polyhedron of a given vector $v$.
		\begin{itemize}
			\item 
			First of all the Meester algorithm produces the Meester continued fraction for $v$ that we convert to the Farey summation continued fraction $\alpha(v)$ by way of Proposition~{\ref{prop: both cf the same}}.
			
			\item 
			Secondly the matrices $M_{i-1}(\alpha(v))$ and $M_i(\alpha(v))$ provide all the vertices of the $i$-th Farey simplex for $v$.
		\end{itemize}
		Here the matrix $M_{i-1}^{-1}(\alpha(v))\cdot M_{i}(\alpha(v))$ tells us which columns are in the simplex. (This matrix is either one of the $A$ or $B$ matrices in position $i$ of the decomposition of Definition~\ref{matrix-cf}.)
	\end{remark}
	
	\begin{remark}
		The matrix multiplication introduced in this subsection can be extended to the higher-dimensional cases in a straightforward way. Since we mostly work in the three-dimension case in what follows, we skip the multidimensional notation here.  
	\end{remark}

	\begin{example}\label{(5,7,8)-part matrices}
		Consider $v=(5,7,8)$ from Example~\ref{(5,7,8)-part 1}.
		As we know from Example~\ref{(5,7,8)-part 2}, the continued fraction for $(5,7,8)$
		is
		\[
		\alpha(v)=[1;1:2:0:0 \, |\, 1].
		\]
		Hence 
		\[
		M(\alpha(v))=A_1A_2A_3^2A_1^0A_2^0B_{31}.
		\]
		It is interesting to note that $v$ can be alternatively obtained 
		from $[1;1:2:0:0 \, |\, 0:1]$.
	\end{example}
	
	\subsection{LLS sequence in the three-dimensional case}
	\label{LLS sequence and the elements}

	In this subsection we discuss the three-dimensional LLS sequence in more detail.
	
	\subsubsection{Values of integer sines in the LLS sequence}

	\vspace{2mm}
	Let us discuss how to get the values of the elements of the Farey continued fraction from LLS sequences of a canonical prismatic flag diagram $T$. 
	Without loss of generality we consider LLS sequence of $\sail_2(T)$, the sail defined by the first and third masts.
	The integer lengths of the mast segments for the first and third masts are shown on the top and the bottom horizontal lines of the sail, which represent masts 1 and 3; they provide nearly two thirds of all the elements of the continued fraction.

	The elements represented by integer lengths of mast segments of mast 2 are not present in our LLS sequence. 
	However, most of these elements are reconstructable from the LLS sequence.
	
	These elements are integer sines of two consecutive faces of $\sail_2(T)$. 
	Such faces are generated by two, three, or four consecutive division tetrahedra (see Definition~\ref{division simplex}). 
	In fact, four consecutive division tetrahedra are needed only in a very specific case when we drop dimension (represented in line $7$ of Table~\ref{duality-table} below); that might happen at most once for every continued fraction.
	
	We have the following different situations for the values of integer sines (without loss of generality we list only the cases when we start with mast 1):
	\begin{longtable}{|l|l|c|}
		\caption{Matrix decomposition and the corresponding values of the LLS sequences.\label{duality-table}}\\
		\hline
		Matrix decomposition & part of the LLS sequences& $\isin$ \\
		\hline       
		$A_1^{a}A_3^{b}$, \quad
		$B_{13}^aB_{31}^b$  & 
		$
		\begin{array}{c}
			\includegraphics[width=0.075\textwidth]{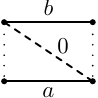}
		\end{array}
		\quad
		\begin{array}{c}
			\includegraphics[width=0.15\textwidth]{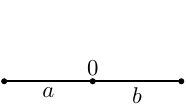}
		\end{array}
		$
		&
		$0$ \\
		\hline
		$A_1^{a}B_{13}^{b}$,
		\quad
		$A_1^{a}B_{31}^{b}$,
		\quad
		$A_1^{a}B_{12}^{b}$
		&
		$
		\begin{array}{c}
			\includegraphics[width=0.15\textwidth]{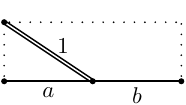}
		\end{array}
		\quad
		\begin{array}{c}
			\includegraphics[width=0.075\textwidth]{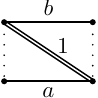}
		\end{array}
		\quad
		\begin{array}{c}
			\includegraphics[width=0.15\textwidth]{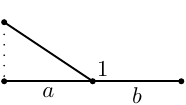}
		\end{array}
		$
		& $1$\\
		\hline
		\hline
		$A_1^{a}A_2^{x}A_1^{b}$, \quad 
		$A_1^{a}A_2^{x}A_3^{b}$
		& 
		$
		\begin{array}{c}
			\includegraphics[width=0.15\textwidth]{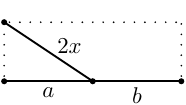}
		\end{array}
		\quad
		\begin{array}{c}
			\includegraphics[width=0.075\textwidth]{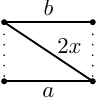}
		\end{array}
		$
		&
		$2x$\\
		\hline
		$B_{12}^{a}B_{21}^xB_{12}^b$
		&
		$
		\begin{array}{c}
			\includegraphics[width=0.15\textwidth]{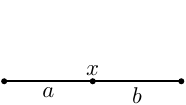}
		\end{array}
		$
		& $x$\\
		\hline
		\hline
		$A_1^{a}A_2^{x}B_{12}^{b}$,
		\quad
		$A_1^{a}A_2^{x}B_{32}^{b}$
		&
		$
		\begin{array}{c}
			\includegraphics[width=0.15\textwidth]{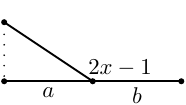}
		\end{array}
		\quad
		\begin{array}{c}
			\includegraphics[width=0.075\textwidth]{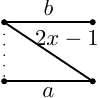}
		\end{array}
		$
		&$2x-1$\\
		\hline
		$A_1^{a}B_{21}^{x}B_{12}^{b}$,
		\quad
		$A_1^{a}B_{23}^{x}B_{32}^{b}$
		&
		$
		\begin{array}{c}
			\includegraphics[width=0.15\textwidth]{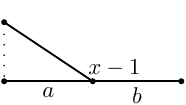}
		\end{array}
		\quad
		\begin{array}{c}
			\includegraphics[width=0.075\textwidth]{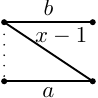}
		\end{array}
		$
		&  $x-1$\\
		\hline
		\hline
		$A_1^aA_2^xB_{21}^yB_{12}^b$,
		\quad 
		$A_1^aA_2^xB_{23}^yB_{32}^b$
		&
		$
		\begin{array}{c}
			\includegraphics[width=0.15\textwidth]{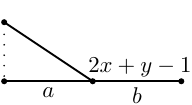}
		\end{array}
		\quad
		\begin{array}{c}
			\includegraphics[width=0.09\textwidth]{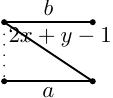}
		\end{array}
		$
		&
		$2x+y-1$\\
		\hline
		$A_1^{a}A_2^{x}B_{13}^{b}$,
		\quad
		$A_1^{a}A_2^{x}B_{31}^{b}$
		&
		$
		\begin{array}{c}
			\includegraphics[width=0.15\textwidth]{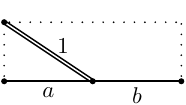}
		\end{array}
		\quad
		\begin{array}{c}
			\includegraphics[width=0.09\textwidth]{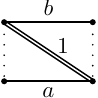}
		\end{array}
		$
		&$1$\\
		\hline
		$A_1^{a}B_{32}^{x}B_{23}^{b}$
		&
		$
		\begin{array}{c}
			\includegraphics[width=0.075\textwidth]{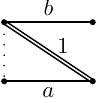}
		\end{array}
		$
		&$1$\\
		\hline
	\end{longtable}
	
	In Table~\ref{duality-table} we show schematically a part of the matrix decomposition for a continued fraction and the corresponding part of the LLS sequence. Here the bottom and the top horizontal lines correspond to masts 1 and 3 respectively.
	Note that the integers $a$, $b$, $x$, and $y$ in the table are assumed to be positive.

	\vspace{2mm}

	In the first two lines of the table there are no division tetrahedra representing the elements of mast~2; here the values on the sail-yards are either $0$ or $1$.
	\vspace{2mm}

	In the next two lines we generate the most common steps when the dimension is not dropped.
	In the third line we get the case of three-dimensional division tetrahedra where the value of the integer sine is $2x$.
	In the fourth line we show the two-dimensional division tetrahedra with the value of the integer sine being equal to $x$.
	This repeats the case of the classical two-dimensional LLS sequence.
	
	\vspace{2mm}
	
	Further in lines~5 and~6 we have two cases where the algorithm loses one dimension but the elements of the continued fractions are still visible on the sail.
	
	\vspace{2mm}
	
	Finally in the last three lines we have the cases when we lose information on one or two of the elements.
	
	\vspace{2mm}
	Let us collect the above observations in the following theorem.
	
	\begin{theorem}\label{3D-LLS-invariance}
		The three-dimensional LLS sequence is an integer invariant 
		of prismatic diagrams \emph{(}and hence, by Theorem~\ref{theorem: prismatic diagrams invariant}, also of Farey polyhedra\emph{)}.
	\end{theorem}
	
	\begin{proof}
		Note that the values of the LLS sequence are completely given by the integer lengths of masts and the combinatorics of prismatic diagrams.
		Here the integer sines of the LLS sequence are written in terms of the opposite mast, as in Table~1 (see also Proposition~\ref{On duality-prop}).
	\end{proof}
	
	\begin{remark}
		Recall that the powers in the matrix form coincide with integer length of the elements in the corresponding masts.
	\end{remark}
	
	\begin{remark}
		The situation is similar in four and higher-dimensions.
		The number of cases will be finite but growing rather fast with the dimension.   
		We omit the exhaustive study of all the cases here.
	\end{remark}

	\subsubsection{Integer arctangent of cones in three-dimensions and sails}
	
	As we have seen, in the last two rows of Table~\ref{duality-table} the LLS sequence does not catch $x$. 
	However the integer invariants of the sail still encode $x$. 
	First we need to set some further integer invariants. 
	
	\vspace{2mm}
	A \textit{cone} is the convex hull of three rays emanating from the same vertex.
	We call the rays the \textit{edges of the cone}.
	A cone is {\it integer} if its vertex is integer. 
	An integer cone is \textit{rational} if each of its rays contain integer points distinct from the vertex.

	\begin{definition}
		
		Let $\alpha$ be a rational cone.
		Assume there is an integer basis in which the vectors of unit integer length that generate the edges of $\alpha$ form (as columns) a matrix
		\[
		\left(
		\begin{matrix}
			1&a_1&b_1\\
			0&a_2&b_2\\
			0&0&b_3\\ 
		\end{matrix}
		\right)
		\]
		satisfying $a_2>a_1>0$, $b_3>b_1\ge 0$, and $b_3>b_2\ge 0$.
		We call this matrix the \textit{integer arctangent} of the cone, and call its elements \textit{integer sines} and \textit{integer cosines}, labelled as follows:
		\[
		\left(
		\begin{matrix}
			1&\icos_{1,2}\alpha&\icos_{1,3}\alpha\\
			0&\isin_1\alpha & \icos_{2,3}\alpha \\
			0&0&\isin_2 \alpha\\ 
		\end{matrix}
		\right).
		\]
	\end{definition}
	
	It turns out that the integer arctangent of a rational cone always exists and is unique. 
	For a proof and further information about integer trigonometry, including higher-dimensional cases, we refer the reader to~{\cite{blackman2023}}.
	
	\begin{proposition}
		There exists a unique integer arctangent for every rational cone.
	\end{proposition}
	
	Note that the integer arctangent of $\alpha$ coincides with the Hermite normal form of the matrix of unit vectors generating the edges of $\alpha$. 
	All the coefficients of this matrix are integer invariants of the cone. 
	
	\subsubsection{Application to sails}
	
	Let us go back to the last three lines of the table. More precisely we illustrate the case $A_1^aA_2^xB_{23}^yB_{32}^b$ of line~7.
	
	\vspace{1mm}
	
	As we see, the integer sine in the LLS sequence provides us the value $2x+y-1$ for the elements $x$ and $y$ (these are the lengths of mast segments of mast $2$ which is not in our sail). 
	Our aim is to construct an integer invariant of the sail that will provide us another equation on $x$ and $y$.
	
	\vspace{2mm}
	
	Consider the simplicial cone $\alpha$ on the vectors generated by:
	
	\begin{itemize}
		\item  the first mast segment of mast $1$; 
		
		\item the yard edge connecting the endpoint of the first segment of mast $3$ with the second vertex of mast $1$;
		
		\item the first mast segment of mast $3$.
	\end{itemize}
	(All these edges are in the sail opposite to the second mast.)
	By definition the arctangent of $\alpha$ is 
	\[
	\left(
	\begin{matrix}
		1 & 0  & x+1\\
		0 & 1  & x+y-1\\
		0 & 0  & 2x+y-1 \\
	\end{matrix}
	\right).
	\]
	From the last column we have:
	\[
	\icos_{1,3}\alpha =x+1, \quad 
	\icos_{2,3}\alpha=x+y-1,\quad 
	\isin_2\alpha=2x+y-1.
	\]
	
	\begin{remark}
		In particular $\isin_2\alpha$ is the value that we have computed for the $LLS$-sequence,
		while the $\icos_{1,3}\alpha$ provides the value of $x$.
	\end{remark}

	\subsection{Continuants}
	\label{subsection: continuants}
	
	In this subsection we briefly discuss the notion of continuants for Farey summation continued fractions.
	
	\vspace{2mm}
	
	Most Jacobi-Perron algorithms are defined by a linear recursion. 
	In the case of the Farey summation algorithm we have the following recursion:
	\begin{equation}\label{recursive-noses}
		\begin{array}{l}
			v_1=(1,0,0); \qquad 
			v_2=(0,1,0); \qquad
			v_3=(0,0,1);\\
			v_{i+3}=v_{i}+a_i(v_{i+1}+v_{i+2})
			\quad \hbox{for $i=1,2,\ldots$}.
		\end{array}
	\end{equation}
	This rule give rise to special functions that are called {\it continuants}.
	
	\vspace{2mm}
	
	In general, most continued fraction algorithms based on matrix multiplication generate their own (natural) notion of continuant.
	It is done as follows.
	
	\vspace{1mm}
	
	Firstly one fixes the generating family of matrices $S(n)$ (or sometimes several families) parametrised by a non-negative integer parameter $n$. In the case of Farey summation continued fractions that is
	\[
	S(n)=
	\left(
	\begin{matrix}
		n & 1 & 0 \\
		n & 0 & 1 \\
		1 & 0 & 0 \\
	\end{matrix}
	\right).
	\]
	Here the matrix $S(n)$ is a composition of $A_1^n$ and a permutation of basis vectors.
	
	\vspace{1mm}
	
	Further we fix the following notation:
	\[
	M_n(a_1,\ldots,a_n)=
	\prod\limits_{i=1}^n S(a_i).
	\]
	Set also
	\[
	v_n(a_1,\ldots,a_n)=M_n(a_1,\ldots,a_n)
	\cdot (1,0,0)^\top.
	\]
	The coefficients of $M_n$ in variables $(a_1,\ldots,a_n)$ are used to determine the continuants. That is suggested by 
	the following two obvious relations on the coefficients of $M_n$:
	\begin{itemize}
		\item
		{\it recursive relation}: 
		$M_n(a_1,\ldots,a_n)=M_{n-1}(a_1,\ldots,a_{n-1})\cdot S(a_n)$;
		\item {\it anti-recursive relation}:
		$M_n(a_1,\ldots,a_n)=S(a_1)M_{n-1}\cdot (a_2,\ldots,a_n)$.
	\end{itemize}
	We synthesise the continuants as follows.

	\begin{definition}
		We define the \textit{$n$-th Farey continuant} iteratively
		\[
		\begin{array}{l}
			K_0=1;\\
			K_1(x_1)=x_1;\\
			K_2(x_1,x_2)=(x_1+1)x_2;\\
			K_n(x_1,\ldots,x_n)=
			x_n\big(
			K_{n-1}(x_1,\ldots,x_{n-1})+
			K_{n-2}(x_1,\ldots,x_{n-2})
			\big)+
			K_{n-3}(x_1,\ldots,x_{n-3}).
		\end{array}
		\]
	\end{definition}
	
	\begin{example}
		Continuants of integer sequences are integers. 
		For instance
		\[
		\begin{aligned}
			K_3(2,3,4)&=45,\\ 
			K_6(15,2,4,32,54,7)&=2800350.
		\end{aligned}
		\]
	\end{example}
	
	\begin{remark}
		Note that, unlike in the two-dimensional case, we do not have the reverse symmetry for continuants:
		\[
		K_n(x_1,\ldots,x_n) \neq K_n(x_n,\ldots,x_1)
		\]
		Instead, from a simple inductive argument we find a separate recursive formula:
		\[
		K_n(x_1,\ldots,x_n)=x_1K_{n-1}(x_2,\ldots,x_n)+x_2K_{n-2}(x_3,\ldots,x_n)+K_{n-3}(x_4,\ldots,x_n).
		\]
	\end{remark}
	
	The recursive and anti-recursive relations imply the following expression for all the coefficients of $M_n$ in terms of continuant functions.
	
	\begin{proposition}
		\label{prop: cf and continuants}
		We have
		\[
		M_n(a_1,\ldots,a_n)=
		\Big(v_n(a_1,\ldots,a_n),
		v_{n-1}(a_1,\ldots,a_{n-1}),
		v_{n-2}(a_1,\ldots,a_{n-2})\Big),
		\]
		where
		\[
		v_i(a_1,\ldots,a_{i})=
		\left(
		\begin{matrix}
			K_{i}(a_1,\ldots,a_{i})\\
			a_1K_{i-1}(a_2,\ldots,a_{i})+
			K_{i-2}(a_3,\ldots,a_{i})\\
			K_{i-1}(a_2,\ldots,a_{i})
		\end{matrix}
		\right).
		\]
		
	\end{proposition}
	
	\begin{proof}
		The expression for $M_n$ in terms of $v_n$, $v_{n-1}$ and $v_{n-2}$ follows from the recursive relation.
		The expression of $v_i$ via continuants is a consequence of the anti-recursive relation. 
	\end{proof}
	
	\begin{remark}
		Note that 
		\[v_n(a_1,\ldots, a_n)=[a_1;\cdots: a_n\, | \,].
		\]
	\end{remark}
	
	\begin{remark}
		\label{remark: stage of continuants}
		The above is entirely related to Stage~1 (namely to matrices $A_1,A_2,A_3$). 
		Stage 2 for $B_{ij}$ employs classical continuants.    
	\end{remark}

	\begin{example}
		Not all integer vectors of unit length are reached by continuants.
		As we have seen before, we sometimes need to use multiplication with $B_{ij}$ matrices
		(see Example~\ref{example-not all tetrahedra}.)
	\end{example}
	
	\begin{proposition}
		Consider an integer vector $p$ with positive coordinates and with unit integer distance to the origin. 
		Let us assume that 
		\[
		p=v_n(a_1,\ldots, a_n)
		\]
		for some non-negative integers $a_i$ 
		$($no two consecutive zeros except possibly at the start, $a_n\ge 2$$)$.
		Then we also have
		\[
		p=v_{n+1}(a_1,\ldots, a_n{-}1,1)=
		v_{n+2}(a_1,\ldots, a_n{-}1,0,1).
		\]
	\end{proposition}
	\begin{proof}
		This follows directly from the fact that the tetrahedra in the tessellation do not intersect.
	\end{proof}

	\begin{remark}
		As we see in the case of existence there are exactly three entirely three-dimensional Farey summation continued fractions representing $p$. The lengths of these continued fractions are three consecutive integers. This is similar to the classical two-dimensional case. In fact a similar statement holds in the multidimensional case (we omit it here).
	\end{remark}

	\subsection{Three-dimensional frieze relation}
	\label{subsection: 3d frieze and ptolemy}
	Frieze patterns are tables of numbers that encode the combinatorics of polygon triangulations.
	In this subsection we say a few words on a construction similar to that of frieze patterns in the classical case. 
	Here we restrict ourselves to the case of three-dimensional Farey summation continued fractions that do not have zero elements.
	We are dealing with the three-dimensional part of the continued fraction only. 
	We also consider only the polyhedra that admit path triangulations.

	\subsubsection{Definition of $\lambda$-lengths}

	The tetrahedra in a prismatic diagram defined by a vector $v$ have a natural ordering. 
	This ordering is defined by the intersection of these tetrahedra with the ray starting at the centre of mass of the deck $R$ and in the direction of $v$. 
	We say that the number of tetrahedra in the prismatic diagram is the \textit{length} of the prismatic diagram.
	
	\begin{definition}
		We say that a polyhedron is an $(i,j)$-\textit{slice} of a prismatic diagram of length $n$ if it is obtained from the diagram by removing all $i-1$ simplices before the $i$-th yard and all $n-j$ simplices after the $j$-th yard. 
		The $i$-th and $j$-th yards are called respectively the \textit{deck} and \textit{nest} of the $(i,j)$-slice.
	\end{definition}
	
	\begin{definition}
		Let $P$ be a prismatic diagram.
		Consider two vertices $(v,w)$ in $P$.
		We say that the {\it geodesic} $\ell(v,w)$ between two vertices of $P$ is the smallest $(i,j)$-slice that contains these two points. 
	\end{definition}

    The next definition generalises Notion~F from Section~\ref{Two-dimensional case}.
	\begin{definition}\label{lambda3-length}
		Since $\ell(v,w)$ is a path-polyhedron it naturally defines a Farey summation continued fraction $F(v,w)$ (as discussed above).
		We say that the continuant of this continued fraction is the \textit{$\lambda$-length} of the geodesic $\lambda(v,w)$. 
		We set $\lambda(v,v)=0$.
		We also set $\lambda(v,w)=1$ if $v\ne w$ and they are connected by a yard edge.
	\end{definition}
	
	\begin{remark}
		In two-dimensions the lambda length between two vertices $(p_1,q_1)$ and $(p_1,q_1)$ of a Farey polygon is as defined in Definition~\ref{definition lambda length}.
	\end{remark}
	
	\subsubsection{Ptolemy relation}
	
	\begin{figure}[t]
		\centering
		\includegraphics[height=4.5cm]{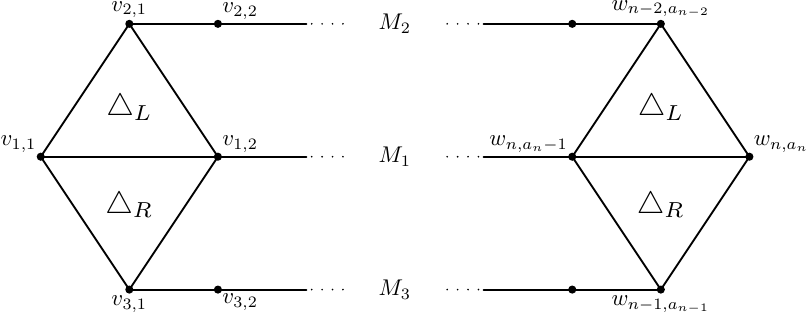}
		\caption{Left ($\triangle_L$) and right ($\triangle_R$) triangles at the beginning and end of the $(1,n)$-slice.}
		\label{figure: L-R triangles}
	\end{figure}
	
	Consider the triangulation of the boundary of the prismatic diagram.
	
	\begin{definition}
		We say that a face $F$ is \textit{nice} if it is not the deck and does not contain a vertex of the nest. 
	\end{definition}
	
	\begin{definition}
		We call a face of a prismatic diagram a \textit{left triangle} ($\triangle_L$) if two vertices are on mast $m$ and one vertex is on mast $m+1\mod3$.
		Similarly we call a face a \textit{right triangle} ($\triangle_R$)  if two vertices are on mast $m$ and one vertex is on mast $m-1\mod3$.
		See Figure~\ref{figure: L-R triangles} for reference.
	\end{definition}

	The \textit{Ptolemy relation} for prismatic diagrams in two-dimensions is defined for pairs of edges of the diagram (cf.~\cite{moriergenoud2019}): if $(a_1,\ldots,a_n)$ is the exponential form of the LR sequence of the minimal polygon containing two edges 
	in the prismatic diagram then
	the Ptolemy relation is the determinant equation
	\[
	\det\left(
	\begin{matrix}
		K_n(a_1,\ldots,a_n) & K_n(a_1,\ldots,a_n-1) \\
		K_n(a_1-1,\ldots,a_n) & K_n(a_1-1,\ldots,a_n-1) \\
	\end{matrix}
	\right)
	=\begin{cases}
		1, & n\text{ even};\\
		-1, & n\text{ odd},
	\end{cases}
	\]
	where $K$ here denotes the two-dimensional continuant.
	The values in the matrix are the $\lambda$-lengths of pairs of vertices of the distinct edges. 
	
	\begin{remark}
		Note that the values of the above matrix are exactly elements of the frieze pattern corresponding to the triangulated polygon.
		The $-1$ occurs precisely when the initial vertex and pennant of the $(i,j)$-th slice are on the same mast.
		This is explained by the difference between our ordering of vertices and the ordering in the classical correspondence between triangulated polygons and frieze patterns (see, for example,~\cite{moriergenoud2019}).
	\end{remark}

	In the three-dimensional case we define the Ptolemy relation for pairs of nice triangles (the faces on the boundary of the prismatic diagram).

	\begin{definition}
		Consider an ordered pair of nice triangles $V=v_1v_2v_3$ and $W=w_1w_2w_3$ not connected by a yard:
		$V$ is ordered clockwise and $W$ is ordered counter-clockwise.
		Let $\Lambda(v_i,w_j)$ denote the $3\times3$ matrix of lambda lengths $\lambda(v_j,w_j)$.
		Then the \textit{Ptolemy constant} for the pair $(V,W)$ is $\det(\Lambda(v_i,w_j))$. 
		Denote it by $P(V,W)$.
	\end{definition}
	\begin{definition}
		We label the vertices of the prismatic diagram in the following way: for the exponential LR sequence $(a_1,\ldots,a_n)$ we label the vertices on the masts by $v_{i,j}$, where $0\leq j< a_i$, and $v_{i,j}$ denotes the $j$-th vertex on the mast segment associated to $a_i$ (as in Definition~\ref{farey continued fractions}).
		
	\end{definition}
	
	Note that vertices connecting the two mast segments $i$ and $j$ are labelled twice, as $v_{i,a_i}$ and $v_{j,1}$.
	We show how to use continuants to calculate lambda lengths.
	
	\begin{proposition}
		Let $(a_1,\ldots,a_n)$ be the exponential form of the LR sequence for $\ell(v_{1,1},w_{n,a_n})$.
		Then for the left and right triangles in $\ell(v_{1,1},w_{n,a_n})$ with vertices $v_{1,1}$ and $w_{n,a_n}$ we have
		\[
		\lambda(v_{i,j},w_{k,l})=
		\begin{cases}
			K(a_i+1-j,a_{i+1},\ldots,a_{k-1},a_k),\ & k\in\{n-1,n-2\},\\
			K(a_i+1-j,a_{i+1},\ldots,a_{k-1},l),\ & k=n.
		\end{cases}
		\]
		\qed
	\end{proposition}
	
	Now we generalise the famous Ptolemy relation to the multidimensional integer setting (Notion~G from Section~\ref{Two-dimensional case}). 
	
	\begin{theorem}\label{Ptolemy-3dd}{\bf (Integer Ptolemy relation.)}
		\label{theorem: ptolemy in 3d}
		Let $V$ and $W$ be nice triangles 
		such that the slices between each pair of vertices, one from $V$ and one from $W$, contain at least one tetrahedron.
		Then we have
		\[
		P(V,W)=
		\begin{cases}
			1,\ &\text{if }V\text{ is a right triangle},\\
			0,\ &\text{if }V\text{ is a left triangle}.
		\end{cases}
		\]
	\end{theorem}

	\begin{proof}
		Let us first consider the case $\Lambda(\triangle_R,\triangle_R)$.
		Assume that the exponential form of the LR sequence for the polyhedron defined by the two triangles is $(a_1,\ldots,a_n)$, so the determinant matrix of $\lambda$-lengths is
		\[
		\Lambda(\triangle_R,\triangle_R)=
		\left(
		\begin{matrix}
			K(a_3,\ldots,a_n) & K(a_3,\ldots,a_n-1) & K(a_3,\ldots,a_{n-1}) \\
			K(a_1,\ldots,a_n) & K(a_1,\ldots,a_n-1) & K(a_1,\ldots,a_{n-1}) \\
			K(a_1-1,\ldots,a_n) & K(a_1-1,\ldots,a_n-1) & K(a_1-1,\ldots,a_{n-1}) \\
		\end{matrix}
		\right).
		\]
		It follows from the definition that the matrix $M_n(a_1,\ldots,a_n)$ has determinant $1$.
		We define two matrices $M_{\text{Row}}$ and $M_{\text{Col}}$ by
		\[
		M_{\text{Row}}=
		\left(
		\begin{matrix}
			0 & 1 & -a_1 \\
			1 & 0 & 0 \\
			1 & 0 & -1
		\end{matrix}
		\right),
		\qquad
		M_{\text{Col}}=
		\left(
		\begin{matrix}
			1 & 1 & 0 \\
			0 & -1 & 1 \\
			0 & -1 & 0
		\end{matrix}
		\right).
		\]
		Then we have that 
		\[
		P(\triangle_R,\triangle_R)=M_{\text{Row}}\cdot M_n(a_1,\ldots,a_n)\cdot M_{\text{Col}}.
		\]
		Since $\det(M_{\text{Row}})=\det(M_{\text{Col}})=1$ we have that $\det\big(\Lambda(\triangle_R,\triangle_R)\big)=1$.
		A similar computation shows that $\det\big(\Lambda(\triangle_R,\triangle_L)\big)=1$.
		Consider now the case 
		\[
		\Lambda(\triangle_L,\triangle_R)
		=
		\left(
		\begin{matrix}
			K(a_1,\ldots,a_n) & K(a_1,\ldots,a_n-1) & K(a_1,\ldots,a_{n-1}) \\
			K(a_2,\ldots,a_n) & K(a_2,\ldots,a_n-1) & K(a_2,\ldots,a_{n-1}) \\
			K(a_1-1,\ldots,a_n) & K(a_1-1,\ldots,a_n-1) & K(a_1-1,\ldots,a_{n-1}) \\
		\end{matrix}
		\right).
		\]
		Using the recursion formula of continuants we see that $\text{Row }1=\text{Row }2+\text{Row }3$, and hence $\det\big(\Lambda(\triangle_L,\triangle_R)\big)=0$.
		Similarly we observe that $\det\big(\Lambda(\triangle_L,\triangle_L)\big)=0$.
	\end{proof}
	
	\begin{remark}
		It is not clear what happens when there are zero elements.
	\end{remark}

	\subsubsection{Frieze pattern in higher-dimensions}
	
	We conclude this subsection with the following topological definition of frieze patterns in three-dimensions.
	
	\begin{definition}
		Consider a Farey polyhedron $P$ and with prismatic diagram $D$ and let $V(D)$ be the set of vertices of $D$. Consider the function
		\[
		\lambda: V(D)\times V(D) \to \mathbb Z, 
		\]
		whose values on two vertices is the $\lambda$-length between the corresponding vertices in the Farey polyhedron.
		We call the collection $(\partial D\times \partial D,\lambda)$ the {\it frieze pattern} associated to the given Farey polyhedron.
	\end{definition}

	As we have discussed above the frieze pattern satisfies the Ptolemy relation for pairs of faces in the prismatic diagram. 
	Let us illustrate this by the following example.
	
	\begin{example}
		\begin{figure}[t]
			\centering
			\includegraphics[width=0.9\textwidth]{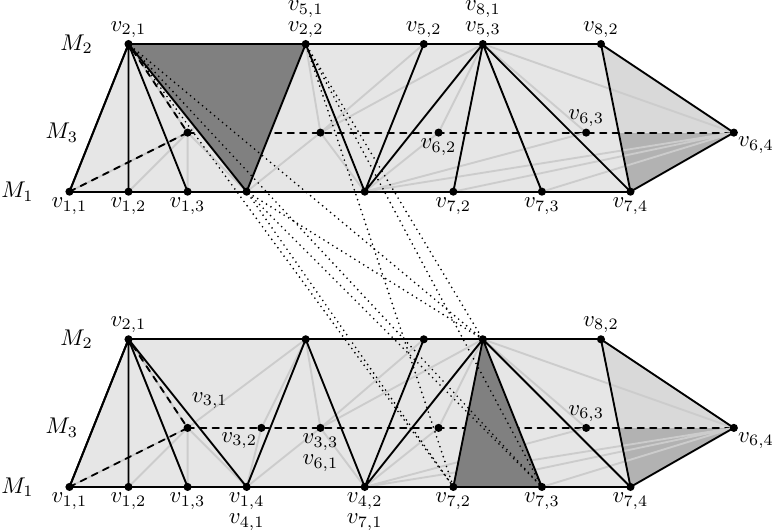}
			\caption{Prismatic diagram for $[3;1:2:1:2:3:3:1]$.}
			\label{figure: face to face}
		\end{figure}
		Consider a continued fraction
		\[
		[3;1:2:1:2:3:3:1]
		\]
		with marked faces $V=v_{2,1}v_{2,2}v_{1,4}$ and $W=v_{7,2}v_{7,3}v_{5,3}$ as in Figure~\ref{figure: face to face}.
		The geodesic $\ell(V,W)$ has exponential LR sequence $(1,2,1,2,3,2)$.
		The pairwise vertices are as follows
		\[
		\begin{aligned}
			P(V,W)&=\left(
			\begin{matrix}
				K(a_1,\ldots,a_n) & K(a_1,\ldots,a_{n-2}) & K(a_1,\ldots,a_n-1) \\
				K(a_1-1,\ldots,a_n) & K(a_1-1,\ldots,a_{n-2}) & K(a_1-1,\ldots,a_n-1) \\
				K(a_3,\ldots,a_n) & K(a_3,\ldots,a_{n-2}) & K(a_3,\ldots,a_n-1) \\
			\end{matrix}
			\right)\\
			&=\left(
			\begin{matrix}
				K(1,2,1,2,3,2) & K(1,2,1,2) & K(1,2,1,2,3,1)\\
				K(0,2,1,2,3,2) & K(0,2,1,2) & K(0,2,1,2,3,1)\\
				K(1,2,3,2) & K(1,2) & K(1,2,3,1)\\
			\end{matrix}
			\right)\\
			&=\left(
			\begin{matrix}
				218 & 21 & 112\\ 
				105 & 10 & 54\\
				41 & 4 & 21
			\end{matrix}
			\right).
		\end{aligned}
		\]
		In particular $\det P(V,W)=1$.
		
	\end{example}

	\begin{remark}
		Informally speaking there is a certain similarity between the structure of multidimensional frieze patterns and Voronoi continued fractions introduced in ~\cite{voronoi1896}: here one constructs a special polyhedron and writes elementary matrix transitions between its faces. The construction of the Voronoi polyhedra are substantially different. The question of establishing the link between Voronoi continued fractions and multidimensional frieze patterns is open.  
	\end{remark}

	\section{Suggestions for further work}
	\label{Plans for further work}
	
	We conjecture that the construction in this paper has a natural extension to
	other subtractive algorithms
	and the multidimensional subtractive algorithms as well. 
	From this we propose the following open problems for further study.
	
	\vspace{2mm}

	\begin{problem}
		Describe the edge sail duality in the higher-dimensional case ($n\ge 4$). Here corresponding tables similar to Table~\ref{duality-table} in the three-dimensional case may be built.   
	\end{problem}

	\begin{problem}
		What is the analogue of nose stretching for other algorithms. What are their geometric properties.    
	\end{problem}

	\begin{problem}
		Extend the theory of Farey simplices and their sails to other subtractive algorithms.
	\end{problem}
	
	\begin{problem}
		Compare the properties of different subtractive algorithms and their characteristic features.
	\end{problem}
	
	\begin{problem}
		Study the frieze pattern defined by prismatic diagrams in three-dimensions and above.
	\end{problem}

	\vspace{2mm}
	\noindent
	{\bf Acknowledgment:} The second author is supported by EPSRC grant EP/W002817/1.
    We are grateful to the anonymous reviewer for their helpful comments.
	
	\printbibliography
	
	\bigskip 
	
	\noindent
	\footnotesize \textbf{Authors' addresses:}
	
	\bigskip

	\noindent{Department of Mathematical Sciences, University of Liverpool,
		UK} 
	\hfill \texttt{karpenk@liverpool.ac.uk}
	\medskip  
	
	\noindent{Department of Computer Science, Mathematics, and Physics, University of West Indies, Cave Hill Campus,
		Barbados} 
	\hfill \texttt{matty.vanson@uwi.edu}
	
\end{document}